\documentclass[12pt]{article}
\usepackage{amsmath, amsfonts,amsthm, amssymb}
\usepackage{multicol}
\usepackage{graphicx}
\usepackage{color,upgreek}
\let\ol=\overline

\let\t=\tilde
\usepackage{epstopdf}
\usepackage{subfigure}
\usepackage[utf8]{inputenc}
\usepackage{hyperref}
\makeatletter
\usepackage{cleveref}

\newcommand{\Aut}{{\rm Aut}}
\newcommand{\Sub}{{\rm Sub}}

\newcommand{\N}{{\mathbb N}}
\newcommand{\Z}{{\mathbb Z}}
\newcommand{\R}{{\mathbb R}}
\newcommand{\Cm}{{\mathbb C}}
\newcommand{\Q}{{\mathbb Q}}

\newcommand{\TT}{\mathbb{T}}

\newcommand{\NN}{\mathcal{N}}

\newcommand{\Id}{{\rm Id}}
\newcommand{\Homeo}{{\rm Homeo}}

\newcommand{\db}{\mathbf{d}}
\newcommand{\Gc}{{\mathcal G}}
\newcommand{\du}{{\rm d}\,}

\newcommand{\inn}{{\rm inn}}
\newcommand{\Ad}{{\rm Ad}}

\newcommand{\GL}{{\rm GL}}
\newcommand{\NC}{{\rm (NC)}}

\newcommand{\rad}{{\rm rad}}
\newcommand{\nil}{{\rm nil}}
\newcommand{\lcm}{{\rm lcm}}
\let\t=\tilde
\let\ol=\overline
\let\Ga=\Gamma
\let\Lam=\Lambda
\let\ga=\gamma

\begin{document}

\title{Dynamics of Actions of Automorphisms of Discrete Groups 
$G$ on Sub$_G$ and Applications to Lattices in Lie Groups
}
\date{}
\author{Rajdip Palit, Manoj B.\ Prajapati  and  Riddhi Shah}

\maketitle

\theoremstyle{plain}
\newtheorem{theorem}{Theorem}[section]
\newtheorem{corollary}[theorem]{Corollary}
\newtheorem{lemma}[theorem]{Lemma}
\newtheorem{proposition}[theorem]{Proposition}
\newtheorem{conjecture}[theorem]{Conjecture}
\newtheorem{hypothesis}[theorem]{Hypothesis}
\newtheorem{condition}[theorem]{Condition}
\newtheorem{fact}[theorem]{Fact}
\newtheorem{problem}[theorem]{Problem}

\theoremstyle{definition}
\newtheorem{definition}[theorem]{Definition}
\newtheorem{notation}[theorem]{Notation}

\theoremstyle{remark}
\newtheorem{remark} [theorem]{Remark}
\newtheorem{remarks}[theorem]{Remarks}
\newtheorem{example}[theorem]{Example}
\newtheorem{examples}[theorem]{Examples}


\begin{abstract}
For a discrete group $G$ and the compact space $\Sub_G$ of (closed) subgroups of $G$ endowed 
with the Chabauty topology, we study the dynamics of actions of automorphisms of $G$ on $\Sub_G$ in terms of distality and 
expansivity. We also study the structure and properties of lattices $\Ga$ in a connected Lie group. In particular, we 
show that the unique maximal solvable normal subgroup of $\Ga$ is polycyclic and the corresponding quotient of $\Ga$ is either 
finite or admits a cofinite subgroup which is a lattice in a connected semisimple Lie group with certain properties. 
We also show that $\Sub^c_\Ga$, the set of cyclic subgroups of $\Ga$, is closed in $\Sub_\Ga$. 
We prove that an infinite discrete group $\Ga$ which is either polycyclic or a lattice in a connected Lie group, does not 
admit any automorphism which acts expansively on $\Sub^c_\Ga$, while only the finite order automorphisms of $\Ga$ act 
distally on $\Sub^c_\Ga$. For an automorphism $T$ of a connected Lie group $G$ and a $T$-invariant lattice $\Ga$ in $G$, 
we compare the behaviour of the actions of $T$ on $\Sub_G$ and $\Sub_\Ga$ in terms of distality. We put certain conditions on 
the structure of the Lie group $G$ under which we show that $T$ acts distally on $\Sub_G$ if and only if it acts distally on $\Sub_\Ga$. 
We construct counter examples to show that this does not hold in general if these conditions on the Lie group are relaxed.  
\end{abstract}
 
 
 
 \section{Introduction}

A homeomorphism $T$ of a (Hausdorff) topological space $X$ is said to be {\it distal} if for every pair of elements 
$x,y\in X$ with $x\ne y$, the closure of $\{(T^n(x), T^n(y))\mid n\in\Z\}$ in $X\times X$ does not intersect the diagonal 
$\{(a,a)\mid a\in X\}$.  If $X$ is compact and metrizable with a metric $d$, then $T$ is distal if and only if given $x,y\in X$ with $x\ne y$, 
$\inf\{d(T^n(x),T^n(y))\mid n\in\Z\}>0$. In case $X$ is a topological group and $T$ is an automorphism, then $T$ is 
distal if and only if the identity of $X$ does not belong to the closure of the $T$-orbit of any nontrivial element in $X$. 
Distal maps on compact spaces were introduced by David Hilbert to study the dynamics of non-ergodic maps. For 
distal actions on compact spaces and on locally compact groups, 
we refer the reader to Abels \cite{Ab78, Ab81}, Ellis \cite{El58}, Furstenberg \cite{Fu63}, Moore \cite{Mo68}, Raja and 
Shah \cite{RS10, RS19}, Shah \cite{Sh10} and the references cited therein.  

For a metrizable topological space $X$ with a metric $d$, a homeomorphism $T$ on $X$ is said to be {\it expansive} if 
there exists $\delta>0$ such that the following holds: If $x,y\in X$ with $x\neq y$, then $d(T^n(x),T^n(y))>\delta$ for some
$n\in\Z$. It is well-known that the expansivity of a homeomorphism on any compact space is independent 
of the metric (cf.\ \cite{Wa82}). Let $G$ be a locally compact Hausdorff group with the identity $e$. An automorphism $T$ of $G$ 
is said to be {\em expansive} if $\cap_{n\in\Z}T^n(U)=\{e\}$ for some neighbourhood $U$ of $e$. 
If $T$ is expansive on $G$, then $G$ is metrizable and the above definition is equivalent to the one given in terms 
of the left invariant metric $d$ on $G$. The notion of expansivity was introduced by Utz \cite{Ut50} and studied 
by many in different contexts (see Bryant \cite{Br60}, Schmidt \cite{Sc95}, Choudhuri and Raja \cite{CR20}, Gl\"ockner and Raja 
\cite{GlR17}, Shah \cite{Sh19} and the references cited therein). It is known that on any compact metric space, the class of 
distal homeomorphisms and that of expansive homeomorphisms are disjoint from each other
\cite{Br60}.
 
Here, we study the dynamics of the actions of automorphisms of certain discrete 
groups $G$ on the compact space of subgroups of $G$ in terms of distality and expansivity. Let $G$ be a locally compact (Hausdorff) 
topological group and let $\Sub_G$ denote the set of all closed subgroups of $G$ equipped with the Chabauty topology \cite{Ch50}. 
Then $\Sub_G$ is compact and Hausdorff. It is metrizable if $G$ is so (see \cite{Ge18} and \S 1 of Ch.\ E in \cite{BP92} for more details). 
Let $\Aut(G)$ denote the group of all automorphisms of $G$. There is a natural action of $\Aut(G)$ on $\Sub_G$; namely, 
$(T,H)\mapsto T(H)$, $T\in\Aut(G)$, $H\in\Sub_G$. For each $T\in\Aut(G)$, the map $H\mapsto T(H)$ defines a homeomorphism
 of $\Sub_G$ (cf.\ \cite{HK14}, Proposition 2.1), and the corresponding map from $\Aut(G)\to\Homeo(\Sub_G)$ is a group 
homomorphism. 

 We say that for a locally compact (metrizable) group $G$, $T\in\Aut(G)$ acts distally (resp.\ expansively) on $\Sub_G$ if 
 the homeomorphism of $\Sub_G$ corresponding to $T$ is distal (resp.\ expansive). The distality of such an action was first studied by 
 Shah and Yadav \cite{SY19} for connected Lie groups and, by Palit and Shah \cite{PaS20} for lattices in certain types of Lie groups. 
 It was shown in \cite{SY19} that if $G$ is such that $G^0$, the connected component of the identity $e$ in $G$, has no nontrivial 
 compact connected central subgroup 
 and if $T$ acts distally on $\Sub_G$, then $T$ is distal. The expansivity of such an action was first studied by Prajapati and 
 Shah \cite{PrS20} for locally compact groups. It was also shown that if $T$ acts expansively on $\Sub_G$, then $T$ is expansive and, 
 $G$ is totally disconnected and it is either finite or noncompact.
 
 Let $\Sub^a_G$ (resp.\ $\Sub^c_G$) denote the space of all closed abelian (resp.\ discrete cyclic) subgroups of $G$. 
 They are invariant under the action of $\Aut(G)$, $\Sub^a_G$ is always closed in $\Sub_G$, while $\Sub^c_G$ is closed for 
 many discrete groups $G$ \cite{PaS20}. In particular, this also holds for any discrete polycyclic group $G$. Here we will show that 
 $\Sub^c_G$ is closed for any lattice $G$ in a connected Lie group. We study distality and expansivity of the actions of 
 automorphisms of a discrete group $G$ on $\Sub^c_G$ when $G$ is discrete and polycyclic or a lattice in a connected Lie group. 
 
 For many groups $G$, the (compact) spaces $\Sub_G$, $\Sub^a_G$ and the closure of $\Sub^c_G$ 
have been identified (see e.g.\ Baik and Clavier \cite{BC13,BC16}, Bridson, de la Harpe and Kleptsyn \cite{BHK09} and 
Pourezza and Hubbard \cite{PH79}). For the 3-dimensional Heisenberg group $\mathbb{H}$, the structure of the space of 
lattices in $\mathbb{H}$ and the action of $\Aut(\mathbb{H})$ on certain subspaces of $\Sub_{\mathbb{H}}$ have 
also been studied in \cite{BHK09}. For a connected Lie group $G$, the space of $G$-invariant measures on $\Sub_G$ and the 
subspace of lattices have also been studied extensively (see Abert et al \cite{Sev17}, Gelander \cite{Ge15} and Gelander and 
Levit \cite{GeL18}). Since the homeomorphisms of $\Sub_G$ arising from the action of $\Aut(G)$ form a 
large subclass of $\Homeo(\Sub_G)$, it is important to study the dynamics of such homeomorphisms of $\Sub_G$.

For a discrete group $\Ga$, which is either polycyclic or a lattice in a connected Lie group, we show that an automorphism of $\Ga$ acts 
 distally on $\Sub^c_\Ga$ if and only if some power of it is the identity map (more generally, see Theorems \ref{poly-d} and  
 \ref{lattice-distal}).
 
  For a lattice $\Ga$ in a connected Lie group $G$, the set $\Sub^c_\Ga$ is usually much smaller than $\Sub^a_G$. In fact, 
  since $\Ga$ is finitely generated, and hence countable, $\Sub^c_\Ga$ is also countable, while $\Sub^a_G$ is not if $G$ is noncompact. 
For $T\in\Aut(G)$ which keeps $\Ga$ invariant, we compare the behaviour of distality of the $T$-actions on $\Sub^c_\Ga$ and  
$\Sub_G$. This was studied in \cite{PaS20} for the cases when $G$ is simply connected nilpotent, simply connected solvable or 
semisimple. We generalise these results of \cite{PaS20} to lattices in connected Lie groups (more generally, see Theorem~\ref{lattice-distal}). 
We also construct counter examples to show that Theorem~\ref{lattice-distal} is the best possible result in this direction. The theorem, in 
particular, shows that for a certain class of connected Lie groups $G$, if an automorphism $T$ of $G$ keeps a lattice $\Ga$ invariant, then 
the distality of the action on $\Sub^c_\Ga$ implies that some power of $T$ is the identity map. 

Some results about distal actions are proven for automorphisms belonging to the class $\NC$ which was introduced in \cite{SY19}. 
For a locally compact metrizable group $G$, 
an automorphism $T\in\Aut(G)$ is said to belong to $\NC$ if for every nontrivial closed cyclic subgroup $A$ of $G$, $T^{n_k}(A)\not\to\{e\}$ 
in $\Sub_G$ for any sequence $\{n_k\}\subset\Z$. The class $\NC$ of automorphisms is studied in details in \cite{SY19} for connected
Lie groups, and in \cite{PaS20} for lattices in certain connected Lie groups. The class $\NC$ is larger than the set of those which 
act distally on $\Sub^a_G$ or the closure of $\Sub^c_G$ as illustrated by Example 3.11 of \cite{PaS20} and Example~\ref{ex1}. 
However, for many groups $G$, it turns out to be the same as the set of those which act distally on $\Sub_G$; see Theorem 4.1 of 
\cite{SY19} and Corollary 3.9 and Theorem 3.16 of \cite{PaS20}, see also Theorem~\ref{lattice-distal}. 

Expansive actions of automorphisms of locally compact (metrizable) group $G$ on $\Sub_G$ are studied in detail in \cite{PrS20}. 
If $G$ is infinite and either compact or connected, then it does not admit automorphisms which act expansively on $\Sub_G$. 
The $p$-adic field $\Q_p$, $p$ a prime, admits automorphisms which act expansively on $\Sub_{\Q_p}$, while the same 
does not hold if $\Q_p$ is replaced by $\Q_p^n$, $n\geq 2$ (Proposition 4.1, \cite{PrS20}). 
Theorem 3.1 of \cite{PrS20} shows that a nontrivial connected Lie group $G$ does not admit any automorphism which acts 
expansively on $\Sub^a_G$. Therefore the question arises whether an automorphism of $G$, which keeps a lattice $\Ga$ invariant, 
could act expansively on $\Sub^a_\Ga$. We show that this also does not happen. More generally, we show that a discrete infinite 
group $\Ga$, which is either polycyclic or a lattice in a connected Lie group, does not admit any automorphism which acts expansively on 
$\Sub^c_\Ga$ (see Theorems \ref{poly-e} and \ref{main-expa}). In particular, Theorem~\ref{main-expa} generalises Theorem 3.1 of 
\cite{PrS20} for the class of automorphisms of a connected Lie group which keep a lattice invariant. 

We get some results on the structure of lattices in connected Lie groups which are useful for the proofs of 
the main results about distal and expansive actions. It is known that any closed subgroup $H$ of a connected Lie group admits a 
unique maximal solvable normal subgroup (say) $H_\rad$ (cf.\ \cite{Ra72}, Corollary 8.6). We show for a lattice 
$\Ga$ in a connected Lie group $G$ that $\Ga_\rad$ is polycyclic and, $\Ga/\Ga_\rad$ is either finite or it admits a subgroup of finite 
index which is a lattice in a connected semisimple Lie group without compact factors and with finite center. Moreover, if $\Ga_\rad$ is finite, 
then $G$ is either compact and abelian or it is an almost direct product of a compact central subgroup and a connected semisimple 
Lie group with finite center; in particular $G$ is compact and abelian or it is reductive (see Proposition~\ref{lattice-rad}).  We also prove an 
elementary but crucial lemma about the structure of $\Sub_G$ for a class of discrete group $G$ with the property that the set of roots 
of $g$ in $G$ is finite for every $g\in G$; the lemma also shows that such a $G$ does not admit any automorphism that acts expansively 
on $\Sub^c_G$ (which is closed) unless it is finite (see Lemma~\ref{nbd}). 

In section 2, we state some basic results and properties of distal and expansive actions. We also describe the topology of $\Sub_G$ 
for locally compact groups $G$. In section 3, we prove some results about the structure of lattices in Lie groups. For the action of an 
automorphism of a discrete group $G$ on $\Sub_G$, where $G$ is either polycyclic or a lattice in a connected Lie group,  
we explore the distality of this action in section 4, while section 5 deals with the study of the expansivity of this action. 

We will assume that all our topological spaces are locally compact, Hausdorff and metrizable. If $G$ is a discrete group or a Lie group,   
then $G$ is metrizable and so is $\Sub_G$. For a topological group $G$ with the 
identity $e$, and a subgroup $H\subset G$, let $H^0$ denote the connected component of the identity $e$ in $H$, $[H, H]$ denote the
commutator subgroup of $H$, $Z(H)$ denote the center of $H$ and let $Z_G(H)$ denote the centraliser of $H$ in $G$. 
Here, $H^0$ and $Z(H)$ are closed characteristic subgroups of $H$, and $[H,H]$ is also characteristic in $H$. 
Also, $Z_G(H)$ is a closed subgroup of $G$ and, it is characteristic in $G$ if $H$ is so. 
In particular, $H^0$, $Z(H)$ and $Z_G(H)$ are normal in $G$, if $H$ is so. For $x\in G$, let $\inn(x)$ denote the 
inner automorphism of $G$ by $x$, i.e.\ $\inn(x)(g)=xgx^{-1}$, $g\in G$, and let $G_x$ denote the cyclic group generated 
by $x\in G$. In case $G_x$ is closed (discrete) in $G$, then $G_x\in\Sub^c_G$. An element $x\in G$ is said to be a torsion element 
if $x^n=e$ for some $n\in\N$. A group $G$ is said to be {\it torsion-free} if it does not have any nontrivial torsion element. 
For any $x\in G$, by convention, $x^0=e$, the identity of the group $G$. Similarly, $T^0=\Id$, 
the identity map, for any bijective map $T$ of a space $X$. For $x\in G$, let $R_x$ denote the set of roots of $x$
in $G$, i.e\ $R_x=\{y\in G\mid y^n=x\mbox{ for some } n\in\N\}$. Note that if $G=\Z^d$, or more generally, if $G$ is a finitely 
generated nilpotent group, then $R_x$ is finite for every $x\in G$ (cf.\ \cite{He77}, Example 3.1.12, Theorems 3.1.13 and 3.1.17).

For a connected Lie group $G$, let $\Gc$ denote the Lie algebra of $G$ and let $\exp:\Gc\to G$ be the exponential map. 
For any $T\in\Aut(G)$, there exists a unique Lie algebra automorphism $\du T:\Gc\to \Gc$ which satisfies 
$\exp(\du T(v))=T(\exp(v))$, $v\in\Gc$. Recall that $\Ad:G\to\GL(\Gc)$, the adjoint representation of $G$ on $\Gc$, is defined as 
$\Ad(g)=\du(\inn(g))$, $g\in G$, and $\Ad(G)$ is a connected Lie subgroup of $\GL(\Gc)$. The radical (resp.\ nilradical) of $G$
is the maximal connected solvable (resp.\ nilpotent) normal subgroup of $G$ and, $G$ is said to be semisimple if its radical is trivial.
A connected Lie group $G$ is said to be reductive if its Lie algebra $\Gc$ is reductive; equivalently, $\Ad(G)$ is semisimple. 
Note that $G$ is reductive if and only if its radical is central in $G$; equivalently, if 
$G$ is an (almost direct) product of a connected semisimple Lie group and $Z(G)$. A connected Lie 
group is said to be linear if it is isomorphic to a subgroup of $\GL(n,\R)$ for some $n\in\N$. We will use certain results 
about the structure of linear groups, Lie groups and Lie algebras which are standard and can be found in any basic 
textbook on Lie groups (see e.g.\ \cite{Ho66, Ho81, Va84}). 

\section{Preliminaries}

 For a (metrizable) topological space $X$, let $\Homeo(X)$ denote the space of all homeomorphisms of $X$. We first state some known 
properties of distal and expansive actions for a compact space $X$. Let $T\in\Homeo(X)$. Then $T^n$ is distal (resp.\ expansive) for some 
$n\in\Z\setminus\{0\}$ if and only if $T^n$ is so for all $n\in\Z\setminus\{0\}$. If $Y\subset X$ is a (nonempty) $T$-invariant 
subspace and if $T$ is distal (resp.\ expansive), then $T|_Y$ is so. If $S\in\Homeo(X)$, then $T$ is distal (resp.\ expansive) 
if and only if $STS^{-1}$ is so. An expansive homeomorphism of $X$ has only finitely many fixed points, and hence, the set of 
its periodic points is countable. If a topological space is discrete, then any homeomorphism is distal as well as 
expansive. The identity map of a space is distal by definition, but it need not be expansive (for example, if the metric space is 
compact and infinite).

Given a locally compact (metrizable) group $G$, the Chabauty topology on $\Sub_G$ was introduced by Chabauty \cite{Ch50}. 
A sub-basis of the Chabauty topology on $\Sub_G$ is given by the sets of the following form 
$O_{K}=\{A\in\Sub_G\mid A\cap K=\emptyset\}$ and $O_{U}=\{A\in\mathrm{ Sub}_G\mid  A\cap U\neq\emptyset\}$, 
where $K$ (resp.\ $U$) is a compact (resp.\ an open) subset of $G$. 

Any closed subgroup of $\R$ is either a discrete group generated by a real number or the whole group $\R$, and 
$\Sub_{\R}$ is homeomorphic to $[0,\infty]$ with a compact topology. 
Any closed subgroup of $\Z$ is of the form $n\Z$ for some $n\in\N\cup\{0\}$, 
and $\Sub_{\Z}$ is homeomorphic to $\{\frac{1}{n}\mid n\in\N\}\cup\{0\}$.
The space $\Sub_{\R^2}$ is homeomorphic to $\mathbb{S}^4$ \cite{PH79}.
Note that the space $\Sub_{\R^n}$ is simply connected for all $n\in\N$ (cf.\ \cite{Kl09}, Theorem 1.3). 

We first state a criterion for convergence of sequences in $\Sub_G$ (cf.\ \cite{BP92}).

For a locally compact first countable (metrizable) group $G$, a sequence $\{H_n\}_{n\in\N}$ in $\Sub_G$ converges to $H$ in 
$\Sub_G$ if and only if the following conditions hold: 
\begin{enumerate}
\item[{(I)}] For any $h\in H$, there exists a sequence $\{h_n\}$ with $h_n\in H_n$, $n\in\N$, such that $h_n\to h$.
\item[{(II)}] For any unbounded sequence $\{n_k\}\subset\N$, if $\{h_{n_k}\}_{k\in\N}$ is such that $h_{n_k}\in H_{n_k}$, $k\in\N$, 
and $h_{n_k}\to h$, then $h\in H$.
\end{enumerate}

In case $G$ is discrete, we know from Lemma 3.2 of \cite{PaS20} that $H_n\to H$ in $\Sub_G$ if and only if 
$H=\cup_{n=1}^\infty\cap_{k=n}^\infty H_k$. In particular, (for such a group $G$) if $H_n\to H$, then $h\in H$ if and only if 
$h\in H_n$ for all large $n$. We will use these criteria for convergence for discrete groups frequently. Also, for a discrete group 
$G$, if each $H_n$ is cyclic and $H_n\to H$, then $H$ is an increasing union of cyclic groups, in particular, one can replace $H_n$ 
by $H'_n=\cap_{k=n}^\infty H_k$ and assume that $H_n\subset H_{n+1}$. 

It is easy to see that $\Sub^a_G$, the set of all closed abelian subgroups of $G$, is closed in $\Sub_G$, but the same need not be 
true for $\Sub^c_G$, the set of discrete cyclic subgroups; e.g.\ $G=\R$. Even if $G$ is discrete, $\Sub^c_G$
need not be closed, e.g.\ $G$ is the group consisting of all roots of unity in the unit circle, endowed with the discrete topology, and
for a prime $p$, the groups $H_n$ of $p^n$th roots of unity are cyclic, $n\in\N$, but $H_n\to\cup_{n\in\N}H_n$, which is not cyclic. 
From now on, for a group $G$, when we say that $\Sub^c_G$ is closed, we mean that $\Sub^c_G$ is closed in $\Sub_G$.
In \cite{PaS20}, for a discrete group $G$, various conditions for $\Sub^c_G$ to be closed are discussed.  
 We now state and prove an elementary lemma which gives one more useful condition involving quotient groups.

\begin{lemma} \label{cyc-sg} Let $G$ be a discrete group and let $H$ be any normal subgroup of $G$. If 
$\Sub^c_H$ and $\Sub^c_{G/H}$ are closed, then $\Sub^c_G$ is closed.  In particular, if $H$ has finite index in $G$ and 
$\Sub^c_H$ is closed, then $\Sub^c_G$ is closed.
\end{lemma}

\begin{proof} If $G$ is finite, then $\Sub_G$ is finite and discrete and hence $\Sub^c_G$ is closed. 
Suppose $G$ is not finite. If $H=\{e\}$ or $H=G$, then the assertions are obvious. 

Now suppose $H$ is a proper subgroup of $G$. Let $\psi: G\to G/H$ be the natural projection. 
Suppose $\Sub^c_H$ and $\Sub^c_{G/H}$ are closed. Let $x_n\in G$, $n\in\N$, be such that $G_{x_n}\to L$ for some 
$L\subset G$.  By Lemma 3.2 of \cite{PaS20}, $L=\cup_n G_n$, 
where $G_n=\cap_{k=n}^\infty G_{x_k}$. In particular, $L$ is an increasing union of cyclic groups $G_n$. Then 
$\psi(L)=\cup_n\psi(G_n)$, an increasing union of cyclic groups, and hence $\psi(G_n)\to \psi(L)$ in $\Sub_{G/H}$. 
As $\Sub^c_{G/H}$ is closed, $\psi(L)$ is cyclic, and hence $\psi(L)=\psi(G_n)$ for all large $n$. 
Therefore, $LH=G_nH$, and hence $L=G_n(L\cap H)$ for all large $n$. Since $L\cap H=\cup_n(G_n\cap H)$, 
$G_n\cap H\to L\cap H$ in 
$\Sub_H$. As each $G_n\cap H$ is cyclic and $\Sub^c_H$ is closed, we have that $L\cap H$ is cyclic, and hence 
that $G_n\cap H=L\cap H$ for all large $n$.
Therefore, we get that $L=G_n$ for all large $n$, and $L$ is cyclic. This implies that $\Sub^c_G$ is closed. 

If $G/H$ is finite, then so is $\Sub_{G/H}$. Therefore, the second statement follows easily from the first.  
\end{proof}

For a locally compact group $G$, recall that there is a natural group action of $\Aut(G)$, the group of automorphisms of $G$, on 
$\Sub_G$ defined as follows:
$$\Aut(G)\times\Sub_G\to\Sub_G;\ (T,H)\mapsto T(H), T\in\Aut(G), H\in\Sub_G.$$
The map $H\mapsto T(H)$ is a homeomorphism of $\Sub_G$ for each $T\in\Aut(G)$ (cf.\ \cite{HK14}, Proposition 2.1), 
and the corresponding map from $\Aut(G)$ to $\Homeo(\Sub_G)$ is a homomorphism. 

If $T$ is an automorphism of $G$ and $H$ is a closed normal 
$T$-invariant subgroup, then $T$ acts distally (resp.\ expansively) on $\Sub_G$ implies that $T$ acts distally (resp.\ expansively) 
on both $\Sub_H$ and $\Sub_{G/H}$ (see Lemma 3.1 in \cite{SY19} and Lemma 2.3 in \cite{PrS20}). However, 
the converse is not true as illustrated by Example 3.2 in \cite{SY19} and Example 4.1 in \cite{PrS20}. In fact, Example 3.2 in \cite{SY19} 
shows that for $G=\R^2$ and a subgroup $H=\R$, $T|_H=\Id$ and $T$ acts trivially on $G/H$, but $T$ does not act distally on 
$\Sub_G$. However, the following elementary but useful lemma shows that in such a case $T$ acts trivially on $G/Z(H)$, and 
in particular that $T=\Id$ if $Z(H)=\{e\}$. We include a short proof for the sake of completeness. 

\begin{lemma} \label{zgh}
Let $G$ be a locally compact group and let $H$ be a closed normal subgroup. Let $T\in\Aut(G)$ be such that $T|_H=\Id$. Then
$T$ acts trivially on $G/Z_G(H)$. In particular, if $T$ acts trivially on $G/H$, then $T$ acts trivially on $G/Z(H)$. 
\end{lemma} 

\begin{proof} Let $g\in G$ and $h\in H$. As $H$ is normal and $T|_H=\Id$, we have that $ghg^{-1}=T(ghg^{-1})=T(g)hT(g^{-1})$. 
Therefore, $g^{-1}T(g)$, centralises every $h\in H$. This implies that $T(g)\in gZ_G(H)$. Moreover, if $T$ acts trivially on $G/H$, then
$T(g)\in g(H\cap Z_G(H))=gZ(H)$, $g\in G$, i.e.\ $T$ acts trivially on $G/Z(H)$.
 \end{proof}

\section{Some properties of lattices in connected Lie groups}

In this section we discuss the structure of lattices in connected Lie groups and prove some useful results. 
Recall that a discrete subgroup $\Gamma$ of a locally compact group $G$ is a lattice in $G$ if $G/\Ga$ carries a 
finite $G$-invariant measure. It is shown by Mostow that any lattice in a connected solvable Lie group $G$ is 
co-compact in $G$ (cf.\ \cite{Ra72}). If $G$ is a simply connected nilpotent Lie group which admits a lattice, then any automorphism 
of the lattice extends to a unique automorphism of $G$ (see Theorem 2.11 and Corollary 1 following it in \cite{Ra72}). 
In general, a lattice need not be co-compact and an automorphism of a lattice need not extend (uniquely) to an automorphism 
of the ambient group. If a connected Lie group admits a lattice then it is unimodular. However, there are connected unimodular Lie 
groups which do not admit a lattice (see e.g.\ Theorem 2.12 and Remark 2.14 in \cite{Ra72}). Any lattice in a connected Lie group
is finitely generated. We now define polycyclic groups, many of which arise as lattices in solvable Lie groups. 

A group $G$ is {\it polycyclic} if it admits a sequence $G=G_{0}\supset G_{1}\cdots\supset G_{k}=\{e\}$ of subgroups such that 
each $G_{i+1}$ is normal in $G_{i}$ and $G_{i}/G_{i+1}$ is cyclic, $0\leq i\leq k-1$. Moreover, if $G_{i}/G_{i+1}$ is infinite, 
$0\leq i\leq k-1$, then $G$ is said to be {\it strongly polycyclic}.

Polycyclic groups are finitely generated and solvable. Every infinite polycyclic group admits a subgroup of finite index which is 
strongly polycyclic. Also, every subgroup of a polycyclic (resp.\ strongly polycyclic) group is polycyclic (resp.\ strongly polycyclic) 
and hence it is finitely 
generated. It is easy to see that a group is polycyclic if and only if it is solvable and every subgroup of it is finitely generated. In particular,
a finitely generated nilpotent group is polycyclic since all its subgroups are finitely generated. 
 It follows from Corollary 3.9 of \cite{Ra72} that any discrete subgroup of a connected solvable Lie group is polycyclic. 
 The following lemma extends this to any discrete solvable subgroup of a connected Lie group. The lemma is known for subgroups 
 of connected solvable Lie groups (cf.\ \cite{Ra72}). It may also be known in general, but we give a proof for the sake of 
 completeness. 

\begin{lemma} \label{solvable-Lie} Let $H$ be a closed solvable subgroup of a connected Lie group. Then 
the following hold:
\begin{enumerate}
\item[{$(1)$}] $H$ is compactly generated,
\item[{$(2)$}] $H/H^0$ is polycyclic and 
\item[{$(3)$}] $H$ admits a normal subgroup $L$ of finite index such that $[L,L]$ is nilpotent.
\end{enumerate}
\end{lemma}

\begin{proof} Let $G$ be a connected Lie group containing $H$. Note that $(1)\implies (2)$ as we have that if $H$ is compactly generated,
then $H/H^0$ is a discrete finitely generated solvable group and, every subgroup of it is also finitely generated since it is a 
quotient of a closed solvable subgroup of $G$. Therefore, $H/H^0$ is polycyclic. Now we prove (1) and  (3). 

 Suppose $G$ is a closed linear Lie group, i.e.\  $G$ is a closed subgroup of $\GL(n,\R)$ for some $n\in\N$.
Let $\t H$ be the Zariski closure of $H$ in $\GL(n,\R)$. Then $\t H$ is solvable and it has finitely many connected components. Hence 
$L=H\cap (\t H)^0$ is a normal subgroup of finite index in $H$ and it is contained in $(\t H)^0$. By Proposition 3.8 of \cite{Ra72}, 
$L/L^0$ is finitely generated, and hence $L$ is compactly generated. This implies that $H$ is compactly generated as $H/L$ is finite. 
Moreover, $[L,L]\subset [(\t H)^0, (\t H)^0]$ which is nilpotent (cf.\  \cite{Ra72}, Proposition 3.11). That is, (1) and (3) hold in this case.  

Let $G$ be any connected Lie group. It is well-known that the center of $G$ is compactly generated, and so are its closed subgroups 
(cf.\ \cite{Mos67}, Theorems 2.4 and 2.6). If $H$ is central in $G$, then $H$ is compactly generated and $(1)$ holds. Now suppose 
$H$ is not central in $G$. Let $Z=:Z(G)$, the center of $G$, and let $\pi: G\to G/Z$ be the natural projection. Then $\pi:G\to \GL(V)$, 
where $V$ is a real vector space isomorphic to the Lie algebra of $G$. 
Given such a representation, there exists a representation $\pi':G\to \GL(V')$ over a suitable vector space $V'$, such that $\ker\pi' 
=\ker\pi=Z$ and $\pi'(G)$ is closed in $\GL(V')$ (see \cite{DM93}, the second paragraph of the proof of Proposition 2.1). 
Therefore, without loss of any generality we may assume that $\pi(G)$ is closed. Then $\ol{\pi(H)}$ is a closed 
solvable subgroup of a closed connected linear group $\pi(G)$. As shown above, $\ol{\pi(H)}$ is compactly generated and 
it has a normal subgroup $L'$ of finite index such that $[L',L']$ is nilpotent. Let 
$L=\pi^{-1}(L')\cap H$. Then $L$ is a normal subgroup of finite index in $H$ and $[L, L]$ is nilpotent, 
since $[\pi(L),\pi(L)]$ is so and $\ker\pi=Z$ is central in $G$. Therefore,
(3) holds. 

To prove (1), it is enough to show that $L$ as above is compactly generated.  Therefore, 
we may replace $H$ by $L$ and assume that $[H,H]$ is nilpotent. 
Let $H'=\ol{HZ}$. Then $H'$ is compactly generated, since $\pi(H')$ and $\ker\pi$ are so. If $H$ is abelian, then so is $H'$, and hence
$H$ itself is compactly generated (cf.\ \cite{Mos67}, Theorem 2.6). Let $H_1=\ol{[H,H]}$ and $H'_1=\ol{[H',H']}$. Then $H'_1=H_1$, and hence 
$H/H_1$ is a closed subgroup of $H'/H_1$ and the later is abelian and compactly generated. Therefore, $H/H_1$ is compactly generated. 
It is enough to show that $H_1$ is compactly generated. Therefore, we may assume that $H$ is nilpotent. Here, $H$ is a closed 
subgroup of $H'=\ol{HZ}$ which is compactly generated and nilpotent. Now it is enough to prove that every subgroup of $H'$ is 
compactly generated. By Lemma 3.1 of \cite{Da10}, $H'$ has a unique maximal compact subgroup (say) $K$. 
Then $H'/K$ is a compactly generated torsion-free nilpotent Lie group and it embeds in a simply connected nilpotent Lie group as a 
closed co-compact subgroup (cf.\ \cite{Ra72}). Therefore, $HK/K$, and hence $H$ is compactly generated. That is, (1) holds. 
\end{proof}

Any closed subgroup $H$ of a connected Lie group admits a unique maximal solvable normal subgroup (cf.\ \cite{Ra72}, Corollary 8.6) 
and we denote it by $H_\rad$. Note that $H_\rad$ is closed and characteristic in $H$. In particular, a lattice $\Ga$ in a connected Lie group 
$G$ admits a unique maximal solvable normal subgroup $\Gamma_\rad$. Moreover, $\Ga_\rad$ is polycyclic by Lemma \ref{solvable-Lie}. 
The following proposition about certain properties of lattices in connected Lie groups will be useful. If $G$ 
is a connected solvable Lie group, then the statements $(b)$ and $(c)$ of the proposition are easy to show. The 
statement $(d)$ below may be known. 

\begin{proposition} \label{lattice-rad} Let $\Gamma$ be a lattice in a connected Lie group $G$. Then 
the following statements hold:
\begin{enumerate}
\item[{$(a)$}] The unique maximal solvable normal subgroup $\Gamma_\rad$ of $\Ga$ is polycyclic. 
\item[{$(b)$}] There exists a unique maximal nilpotent normal subgroup $\Ga_\nil$ in $\Ga$. If $\Ga_\nil$ is finite, then $\Ga_\rad$
is finite. 
\item[{$(c)$}] If $\Gamma_\rad$ is finite, then the following hold: The radical of $G$ is compact and central in $G$. The group $G$ is 
either compact and abelian, or $G$ is reductive and it is an almost direct product of a compact group and a semisimple group with finite 
center.  Moreover, $G$ admits a finite central subgroup $F$ such that $G/F$ is linear. 
\item[{$(d)$}] If $G$ is semisimple and has no compact factors, then $\Gamma_\rad=Z(\Ga)\subset Z(G)$ and it is a subgroup of
finite index in $Z(G)$.  
\end{enumerate}
\end{proposition}

\begin{proof} 
\noindent{$(a):$} This follows from Lemma~\ref{solvable-Lie}\,(2). \\

\noindent{$(b):$} Any nilpotent normal subgroup of $\Ga$ is contained in $\Ga_\rad$. As $\Ga_\rad$ is polycyclic, 
by Corollary 2 to Lemma 4.7 in \cite{Ra72}, $\Ga_\rad$ has a 
unique maximal nilpotent normal subgroup, which is also the unique maximal nilpotent normal subgroup of $\Ga$. 

Now suppose $\Ga_\nil$ is finite. We show that $\Ga_\rad$ is also finite. If possible suppose $\Ga_\rad$ is infinite. 
Since $\Ga_\rad$ is polycyclic, it has a normal subgroup $\Ga'$ of finite index which is strongly polycyclic and 
$[\Ga',\Ga']$ is nilpotent (see Lemma 4.6 and Corollary 4.11 of \cite{Ra72}). Since $[\Ga',\Ga']\subset \Ga_\nil$ which is finite, and the 
former is torsion-free, we get that $[\Ga',\Ga']$ is trivial and hence $\Ga'$ is abelian. Therefore, $\Ga'\subset\Ga_\nil$ is finite. 
Since $\Ga'$ is torsion-free, it is trivial. This implies that $\Ga_\rad$ is finite. \\

\noindent{$(c):$} Now suppose $\Ga_\rad$ is finite. We first show that the radical $R$ of $G$ is compact. 
Suppose $G$ is solvable. Then $G=R$ and 
$\Ga=\Ga_\rad$, and it is finite. By Theorem 3.1 of \cite{Ra72}, $G$ is compact. Now $G$ is abelian (cf.\  \cite{Iw49}, Lemma 2.2). 
In this case, $G$ itself is linear (see Theorem 3.2 of ChXVIII in \cite{Ho66}). Now suppose $G$ is not solvable. 
Let $G=SR$ be a Levi decomposition, where $S$ is a Levi subgroup of $G$. Note that $S$ is a connected semisimple 
Lie subgroup. 

Let $K$ be the maximal compact normal subgroup of $G$. Suppose $K^0$ is trivial. Then we show that $S$ does not admit a 
nontrivial compact factor which centralises $R$. If possible, suppose $S$ has a nontrivial compact factor (say) $C$ which 
centralises $R$. As $C$ is normal in $S$ and it centralises $R$, it is normal in $G$. 
This implies that $C\subset K^0$, which leads to a contradiction. Therefore, $S$ does 
not have any compact factor centralising the radical $R$. By Corollary 8.28 of \cite{Ra72}, $\Ga\cap R$ is a lattice in $R$ and 
it is normal in $\Ga$. Therefore, $\Ga\cap R\subset \Ga_\rad$, and hence it is finite. This implies that $R$ is compact since 
$\Ga\cap R$ is co-compact in $R$ (cf.\  \cite{Ra72}, Theorem 3.1). 

Suppose $K^0$ is nontrivial. As $K$ is compact and normal in $G$, $\Gamma\cap K$ is a finite normal subgroup of $\Ga$. Let 
$\Delta:=Z_\Ga(\Ga\cap K)$. Then $\Delta$ is a normal subgroup of finite index in $\Ga$, and hence it is a 
lattice in $G$. Moreover, $\Delta\cap K$ is finite and central in $\Delta$. From $(a)$, we get that $\Delta$ admits a unique 
maximal solvable normal subgroup $\Delta_\rad$, which is characteristic in $\Delta$, and hence normal in $\Ga$. Therefore, 
$\Delta_\rad\subset \Ga_\rad$ and hence it is finite. Let $\pi:G\to G/K$ be the natural projection. 
Then $R':=\pi(R)$ is the radical of $\pi(G)$. Since $\pi(G)$ has no nontrivial compact normal subgroup,
$\pi(\Delta)\cap R'$ is a lattice in $R'$ and it is normal in $\pi(\Delta)$. As $\ker\pi\cap\Delta$ is central in $\Delta$, we get 
that $\Delta\cap \pi^{-1}(R')=\Delta\cap KR$ is a solvable normal subgroup of $\Delta$, and hence it is finite. Now 
$\pi(\Delta)\cap R'$ is finite, and being a lattice in $R'$, it is cocompact in $R'$. Therefore, $R'$ is compact. Since $\pi(G)$ 
has no nontrivial compact normal subgroups, $R'=\pi(R)$ is trivial, and hence $R\subset K$ and it is compact. 

Now $R$ is compact, and by Lemma 2.2 of \cite{Iw49}, $R$ is abelian. As $R$ is normal in $G$, 
and the latter is connected, by Theorem 4 of \cite{Iw49}, $R$ is central in $G$. 
For the Lie algebra $\Gc$ of $G$ and $\Ad:G\to\GL(\Gc)$, we have that 
$\Ad(G)=\Ad(S)$ is semisimple and $G$ is reductive. Also, $G$ is an almost direct product of 
$S$ and $R$, and $S\cap R$ is central in $S$ as well as in $R$, as the latter is abelian. Therefore, $S\cap R$ is 
central in $G$. Let $S=K'S'$, where $K'$ 
is a product of all compact factors of $S$ and $S'$ is a connected semisimple Lie group without compact factors. Then $K'R\subset K^0$.
Since $K\cap S'$ is normal in $S'$, it is closed in $S'$ (\cite{Rag72}). Therefore, $K\cap S'$ is compact, and hence it is finite and 
central in $S'$ (as $S'$ has no compact factors). In particular, $K'R=K^0$ and $G=KS'=K^0S'$, an almost direct product. 

Now we show that the center of $\pi(S')$ is finite, where $\pi:G\to G/K$ as above. This would also imply that $Z(S')$ is finite 
and also that $Z(S)$ is finite. We know that $\Delta\cap K\subset Z(\Delta)$ and 
$\pi(\Delta)$ is a lattice in $\pi(G)=\pi(S')$. Therefore, $\pi(\Delta)_\rad=\pi(\Delta_\rad)\subset \pi(\Ga_\rad)$ is finite, and hence 
$Z(\pi(\Delta))$ is finite. Since $\pi(S')$ is a semisimple group without compact factors, 
by Corollary 5.18 of \cite{Ra72}, $Z(\pi(\Delta))\subset Z(\pi(S'))$. Moreover, by Theorem 5.17 of \cite{Ra72}, $Z(\pi(S'))\pi(\Delta)$ is
discrete, and hence $Z(\pi(\Delta))$ is a subgroup of finite index in $Z(\pi(S'))$. This implies that $Z(\pi(S'))$ is finite. 
As $\ker\pi\cap S'=K\cap S'$ is a normal subgroup of $S'$, we get that it is closed in $S'$ \cite{Rag72}. 
Therefore, it is a compact normal subgroup of $S'$ and, since $S'$ has no compact factors, it is finite, 
This, together with the fact that $\pi(Z(S'))\subset Z(\pi(S'))$ is finite, imply that $Z(S')$ is finite. Therefore, $Z(S)$ is finite. 
For $\Ad:G\to\GL(\Gc)$ as above. $\Ad(S)$ is closed in $\GL(\Gc)$ and it is isomorphic to $S/(S\cap Z(G))$. 
Let $F=S\cap Z(G)$. Then $F$ is finite and $S/F$, being isomorphic to $\Ad(S)$, is linear. Now the radical of 
$G/F$ is compact and abelian, and the Levi subgroup of $G/F$ is linear. Hence by Theorems 3.2 and 4.2 of ChXVIII in \cite{Ho66}, 
$G/F$ is linear. \\

\noindent{$(d):$} Let $G$ be a semisimple group without compact factors. As observed above, $Z(\Ga)\subset Z(G)$ and it
is a subgroup of finite index in $Z(G)$ (cf.\ \cite{Ra72}, Theorem 5.17 and Corollary 5.18). Let $\psi: G\to G/Z(G)$ be the 
natural projection, where $Z(G)$, the center of $G$, is discrete. Then for some $n\in\N$, $\psi(G)$ is a closed subgroup of $\GL(n,\R)$ 
which is almost algebraic (i.e.\ $\psi(G)$ is a subgroup of finite index in an algebraic subgroup of 
$\GL(n,\R)$), and $\psi(G)$ has trivial center. By Theorem 5.17 of \cite{Ra72}, $Z(G)\Ga$ is closed, and hence 
$\psi(\Ga)$ is a lattice in $\psi(G)$. 
Let $L$ be the Zariski closure of $\psi(\Ga_\rad)$ in $\GL(n,\R)$. Then $L$ has finitely many connected 
components and, $L^0\subset \psi(G)$ as $\psi(G)$ is almost algebraic. Moreover, $\psi(\Ga)$ normalises $L$, 
and hence it normalises $L^0$. As $G$ has no compact factors, $\psi(\Gamma)$ is Zariski dense in $\psi(G)$.  
Hence, the preceding assertion implies that $\psi(G)$ normalises both $L$ and $L^0$. Now $L^0$ is a connected  
solvable normal subgroup of $\psi(G)$ and the latter is semisimple. Therefore, $L^0$ is trivial, and hence $L$ as well as 
$L\cap\psi(G)$ are finite. Since $L\cap\psi(G)$ is also normal in $\psi(G)$, it is central in $\psi(G)$. Therefore
$L\cap\psi(G)$ is trivial, hence we get that $\Ga_\rad\subset Z(G)$, and that $\Ga_\rad=Z(\Ga)$ is a subgroup of finite 
index in $Z(G)$. 
\end{proof}

The following proposition about certain aspects of the structure of lattices in connected Lie groups will be very useful in 
proving the main results about expansivity and distality. 

\begin{proposition} \label {quo-poly}
Let $\Ga$ be a lattice in a connected Lie group $G$ and let $\Ga_\rad$ be the unique maximal solvable normal subgroup of $\Ga$. 
Then the following hold:
\begin{enumerate}
\item[{$(1)$}] Either $\Ga/\Ga_\rad$ is finite or $\Ga$ admits a normal subgroup $\Lambda$ of finite index such that 
$\Ga_\rad\subset\Lam$ and $\Lam/\Gamma_\rad$ is a lattice in a connected semisimple Lie group $G'$ with finite center and 
without compact factors, where $G'$ is a quotient group of $G$.  
\item[{$(2)$}] $\Ga/\Ga_\rad$ admits a subgroup $\Ga'$ of finite index which is torsion-free and it has the property that the set 
$R_g$ of roots of $g$ in $\Ga'$ is finite for all $g\in \Ga'$. 
\end{enumerate}
\end{proposition}

\begin{proof} {\bf $(1):$} {\bf Step 1:} Suppose $G$ is a connected semisimple Lie group without compact factors. By 
Proposition~\ref{lattice-rad}\,$(d)$, $\Ga_\rad=Z(\Ga)\subset Z(G)$ and it is a subgroup of finite index in $Z(G)$. 
Then $G/\Ga_\rad$ is a connected semisimple Lie group without compact
factors. Moreover, it has finite center, since $Z(G/\Ga_\rad)=Z(G)/\Ga_\rad$ (the latter statement follows from the fact that the 
center of any connected semisimple Lie group is discrete). Let $G'=G/\Ga_\rad$. Then $\Ga/\Ga_\rad$ is a lattice in $G'$. \\

\noindent{\bf Step 2:} We now note a useful general statement: For any closed normal subgroup 
$H$ of $G$ and a natural projection $\rho:G\to G/H$, if  $\rho(\Ga)$ is a lattice in $\rho(G)$ and $\Ga\cap H$ is solvable, then 
$\rho(\Ga_\rad)=\rho(\Ga)_\rad$. One way inclusion $\rho(\Ga_\rad)\subset\rho(\Ga)_\rad$ is obvious. The equality follows
from the facts that $\Ga\cap H$, the kernel of $\rho|_\Ga$, is solvable and normal, which in turn implies that the inverse image of 
any solvable normal subgroup of $G/H$ under $\rho|_\Ga$ is solvable and normal in $\Ga$. \\

\noindent{\bf Step 3:} Let $G$ be any connected Lie group. If $G$ is solvable, then so is $\Ga$, and hence $\Ga=\Ga_\rad$ and (1) holds. 
Let $K$ be the unique maximal compact connected normal subgroup of $G$ 
and let $\Ga_1=Z_\Ga(\Ga\cap K)$. Then $\Ga_1$ is a normal subgroup of finite index in $\Ga$ and, 
$\Ga_1\cap K$ is central in $\Ga_1$. Hence $\Ga_1\cap K$ is normal in $\Ga$ and it is contained in $\Ga_\rad$. Let 
$\Ga_2 =\Ga_1\Ga_\rad$ and let $F=\Ga_2\cap K$. We claim that $F$ is solvable. 

Let $\chi:\Ga_2\to\Ga_2/\Ga_1$ be the natural projection.
Then $\chi(\Ga_2)=\chi(\Ga_\rad)$ is solvable. Since $\ker\chi\cap F=\Ga_1\cap K$ is abelian, it follows that $F$ is solvable. 
Replacing $\Ga_1$ by $\Ga_2$ we have that $\Ga_\rad\subset\Ga_1$ and $\Ga_1\cap K$ is a finite solvable group 
which is normal in $\Ga$, and contained in $\Ga_\rad$. Moreover, $(\Ga_1)_\rad$, being characteristic 
in $\Ga_1$, is normal in $\Ga$. Therefore, $(\Ga_1)_\rad=\Ga_\rad$. 

Let $\pi:G\to G/K$ be the natural projection. Since $\ker\pi\cap \Ga_1=K\cap \Ga_1$ is solvable, it follows from Step 2 
that $\pi(\Ga_\rad)=\pi(\Ga)_\rad$. We also have that $(\Ga_1\cap K)_\rad=\Ga_\rad\cap K$. 
Observe that $\pi(R)$ is the radical of $G/K$. Since $G/K$ has no nontrivial compact normal subgroup, it follows from 
Corollary 8.28 of \cite{Ra72} that $\pi(\Ga_1)\cap \pi(R)$ is a lattice in $\pi(R)$. Since $\pi(\Ga_1)\cap\pi(R)$ is solvable and 
normal in $\pi(\Ga)$, we get that $\pi(\Ga_1)\cap\pi(R)\subset \pi(\Ga_\rad)$ and 
$\pi^{-1}(\pi(\Ga_1)\cap\pi(R))\cap\Ga_1\subset \Ga_\rad$. \\

\noindent{\bf Step 4:} Let $\pi_1:G\to G/KR$ be the natural projection, where $R$ is the radical of $G$. Since $G/K$ 
has no nontrivial compact normal subgroup, arguing as in the proof of Proposition~\ref{lattice-rad}\,$(c)$, we get from 
Corollary 8.28 of \cite{Ra72} that $\pi_1(\Ga_1)$ is a lattice in $G/KR$ which is semisimple. 
Let $K'$ be the product of all compact (simple) factors of $G/KR$ (we choose $K'$ to be trivial if $G/KR$ has no compact factors). 
Then $K'$ is normal in $\pi_1(G)$. Moreover, $\pi_1(\Ga)\cap K'$ is finite. Arguing as in Step 3, we get that $\pi_1(\Ga)$, and 
hence $\Ga$ has a normal subgroup 
$\Lam_1$ of finite index such that $\pi_1(\Lam_1)\cap K'$ is central in $\pi_1(\Lam_1)$. Let $\Lam=\Lam_1\cap\Ga_1$. Then $\Lam$ is
normal in $\Ga$ as well as in $\Ga_1$. As $\pi_1(\Ga_\rad)\subset \pi_1(\Ga)_\rad$, arguing as in Step 3, we get that 
$\pi_1(\Lam\Ga_\rad)\cap K'$ is a solvable normal subgroup of $\pi_1(\Ga)$. 

We know from the latter part of Step 3 that $\pi^{-1}(\pi(\Ga_1)\cap \pi(R))\cap\Ga_1\subset\Ga_\rad$, i.e. 
$(\Ga_1\cap RK)K\cap \Ga_1=
(\Ga_1\cap KR)(K\cap \Ga_1)\subset\Ga_\rad$. Therefore, $\Ga_1\cap KR\subset \Ga_\rad$. This implies in particular that 
$\ker\pi_1\cap\Lam\Ga_\rad\subset KR\cap \Ga_1\subset\Ga_\rad$ and $\ker\pi_1\cap\Lam\Ga_\rad$ is a solvable normal subgroup 
in $\Ga$ as well as in $\Ga_1$. 
Replacing $\Lam$ by $\Lam\Ga_\rad$ and arguing as in Step 3, we get that $\Lam\cap K$ (resp.\ $\pi_1(\Lam)\cap K'$) is a solvable 
normal subgroup of $\Ga$ as well as of $\Ga\cap K$ (resp.\ $\pi_1(\Ga)$ as well as of $\pi_1(\Ga)\cap K'$). \\

\noindent{\bf Step 5:} Note that $\Lam$ is a lattice in $G$, it is normal in $\Ga$ and $\Ga_\rad\subset\Lam$, and hence 
$\Ga_\rad=\Lam_\rad$. 
As observed above, $\Lam\cap\ker\pi_1=\Lam\cap KR$ is solvable, hence it follows from Step 2 that $\pi_1(\Lam_\rad)=\pi_1(\Lam)_\rad$. 
Also, $\pi_1(\Lam)\cap K'$, being solvable and normal in $\pi_1(\Ga)$, is contained in $\pi_1(\Lam_\rad)$. 
Therefore, by Step 2, we have that $\pi_1^{-1}(K')\cap\Lam\subset\Lam_\rad$. 
Let $L=\pi_1^{-1}(K')$. Then $L$ is a closed normal subgroup of $G$, $KR\subset L$, 
$L/R$ is compact. Moreover, either $G=L$ or $G/L$ is a connected semisimple Lie group without compact factors. We also have that
$L\cap\Lam\subset\Lam_\rad$ which is solvable. 

If $G=L$, then $\Lam$ is solvable, $\Lam=\Ga_\rad$ and $\Ga/\Ga_\rad$ is finite and (1) holds in this case. 
Now suppose $G\ne L$. Let $\pi_2:G\to G/L$ be the natural projection. 
Since $\pi_2(G)=\pi_1(G)/K'$ and $\pi_1(\Lam)$ is a lattice in $\pi_1(G)$, we get that both $\pi_2(\Ga)$ and 
$\pi_2(\Lam)$ are lattices in $\pi_2(G)$. By Step 2, we have that $\pi_2(\Lam_\rad)=\pi_2(\Lam)_\rad$. 

As shown in Step 1 above (see also the proof of Proposition~\ref{lattice-rad}\,$(d)$), $\pi_2(\Lam)_\rad=Z(\pi_2(\Lam))\subset Z(\pi_2(G))$. 
Therefore, $\pi_2(\Lam_\rad)=Z(\pi_2(\Lam))$ which is central in $\pi_2(G)$. 
Let $M=\pi_2^{-1}(Z(\pi_2(\Lam)))$. Then $M=\Lam_\rad\ker\pi_2=\Lam_\rad L$. This, together with the fact that $L\cap \Lam$ is solvable, 
implies that $M\cap\Lam=\Lam_\rad(L\cap \Lam)=\Lam_\rad$. Now $G/M$ is isomorphic to $\pi_2(G)/\pi_2(\Lam_\rad)$, which is a connected 
semisimple group with finite center and without compact factors. 
Moreover, $(\Lam M)/M$ is a lattice in $G/M$ and it is isomorphic to $\Lam/\Lam_\rad$. As $\Lam_\rad=\Ga_\rad$, we get that 
$\Lam/\Ga_\rad$ is also a lattice in $G/M$. Let $G'=G/M$. Then (1) holds. \\

\noindent{\bf $(2):$} If $\Ga/\Ga_\rad$ is finite, then we can take $\Ga'$ to be trivial and the assertion follows immediately. Now suppose 
$\Ga/\Ga_\rad$ is infinite. From (1), there exists a normal subgroup $\Lam$ of finite index in $\Ga$ such that $\Lam_\rad=\Ga_\rad$ and 
$\Lam/\Ga_\rad$ is a lattice in a connected semisimple Lie group $G'$ with finite center and without 
compact factors. Let $\Lam'=\Lam/\Ga_\rad$ and let $\psi:G'\to G'/Z(G')$ be the natural projection, where $Z(G')$ is the center of $G'$. 
As $Z(G')$ is finite, $\psi(\Lam')$ is a lattice in $\psi(G')$. Since $G'/Z(G')$ is linear, we get from Selberg's Lemma that
$\psi(\Lam')$ admits a torsion-free subgroup (say) $\Lam''$ of finite index. Since $(\Lam')_\rad=\{e\}$, we have that 
$\Lam'\cap Z(G')=\{e\}$, and hence $\psi^{-1}(\Lam'')\cap\Lam'$ is a torsion-free subgroup of finite index in $\Lam'$. 
This, together with Lemma 3.13 of \cite{PaS20}, implies that 
$\Lam'=\Lam/\Ga_\rad$ admits a torsion-free subgroup (say) $\Ga'$ of finite index such that the set $R_g$ of roots of $g$ in $\Ga'$ 
is finite for every $g\in\Ga'$. Since $\Lam$ is a subgroup of finite index in $\Ga$, we have that $\Ga'$ is a subgroup of
finite index in $\Ga/\Ga_\rad$. 
\end{proof}

For a lattice $\Ga$ in a connected Lie group $G$, it is shown in \cite{PaS20} that $\Sub^c_\Ga$ is closed in $\Sub_\Ga$ if 
$G$ is either solvable or semisimple. Here, we generalise this to lattices in any connected Lie group.

\begin{corollary} \label{cyc-cl} 
Let $G$ be a connected Lie group and let $\Gamma$ be a lattice in $G$. Then $\Sub^c_\Gamma$ is closed. 
\end{corollary}

\begin{proof} By Lemma~\ref{solvable-Lie}\,(2), $\Ga_\rad$ is polycyclic, every subgroup of it is finitely generated and, by  
Lemma 3.3 in \cite{PaS20}, $\Sub^c_{\Ga_\rad}$ is closed. From Proposition~\ref{quo-poly}\,(1), we have that $\Ga$ has a 
normal subgroup $\Lam$ of finite index such that $\Ga_\rad\subset\Lam$ and $\Lam/\Ga_\rad$ is either finite or it is a lattice 
in a connected semisimple Lie group.  
In the first case, $\Sub^c_{\Lam/\Ga_\rad}$ is finite, and in the second 
case, it is closed by Lemma 3.14 of \cite{PaS20}. As $\Ga/\Lam$ is finite, so is 
$\Sub_{\Ga/\Lam}$. Now by Lemma~\ref{cyc-sg}, $\Sub^c_\Gamma$ is closed. 
\end{proof}

\section{Distal actions of automorphisms of lattices ${\mathbf \Gamma}$ of Lie Groups on Sub$_{\mathbf \Gamma}$}

In this section we discuss and characterise automorphisms of a lattice $\Ga$ which are in class $\NC$ and 
also those which act distally on $\Sub^c_\Ga$. For a certain class of connected Lie groups $G$, we characterise those automorphisms 
of $G$ which keep a lattice $\Ga$ invariant, and act distally on $\Sub^c_\Ga$.  We first state and prove some lemmas for discrete groups. 
Note that any finitely generated abelian group has a unique maximal finite subgroup. 
 
 \begin{lemma} \label{NC}
 Let $G$ be a discrete group and let $T\in\Aut(G)$ be such that $T\in\NC$. Let $H$ be a closed normal 
 subgroup of $G$ such that $T|_H=\Id$. Let 
 $x$ be a nontrivial torsion element in $G$ such that $T(x)\in xH$. Then there exist $l,n\in\N$ such that $x^l\ne e$ and 
 $T^n(x^l)=x^l$. 
 
 Moreover, if $Z(H)$ is finitely generated, then there exists $m\in\N$ such that $m$ is the order of the unique 
 maximal finite group of $Z(H)$ and the following holds: For every nontrivial torsion element $x\in G$ with $T(x)\in xH$,
there exists $l\in\N$, which depends on $x$, such that $T^m(x^l)=x^l\ne e$. 
\end{lemma}

\begin{proof} Let $x\in G$ be a nontrivial torsion element and let $n_x\in\N\setminus\{1\}$ be the smallest number such that $x^{n_x}=e$. 
Then $T(x^j)=x^jy_j$ for some $y_j\in H$, $1\leq j<n_x$. If $y_j=e$ for some $j$ with $1\leq j<n_x$, then $T(x^j)=x^j\ne e$ and the first 
assertion holds for $l=j$ and $n=1$. Suppose $y_j\ne e$ for all $j$, $1\leq j<n_x$. Now we show that $y_l$ has finite 
order for some $l$ such that $1\leq l<n_x$. Since $T\in\NC$ and $\Sub_G$ is compact, $T^{n_k}(G_x)\to A\ne\{e\}$ for some 
unbounded monotone sequence $\{n_k\}\subset\N$. Let $a\in A$ be such that $a\ne e$. Since $G$ is discrete and $x^{n_x}=e$, 
passing to a subsequence of $\{n_k\}$ if necessary, we get that for all $k$, $T^{n_k}(x^l)=a$, for some fixed $l$ such that 
$1\leq l<{n_x}$. Therefore, $x^ly_l^{n_k}=a$ and hence $y_l^{n_k}=x^{-l}a$ for infinitely many $k$. 
 This implies that $y_l$ has finite order (say) $n$. Since $T(x^l)=x^ly_l$, we get that $T^n(x^l)=x^ly_l^n=x^l$.
 
Now suppose $Z(H)$ is finitely generated. To prove the last statement, we may assume that 
$G=\{g\mid T(g)\in gH\}$. By Lemma~\ref{zgh}, we get that $T$ acts trivially on $G/Z(H)$. Let $m$ be the order of the 
unique maximal finite subgroup (say) $F$ of $Z(H)$. We get as above that for a torsion element $x\in G$, there exists $l\in\N$,  
which depend on $x$, such that $x^l\ne e$ and $T(x^l)=x^ly_l\in x^lF$. Therefore, $T^m(x^l)=x^ly_l^m=x^l\ne e$. 
\end{proof}

The following corollary will be used often and it is an easy consequence of Lemma~\ref{NC}. We give a short proof for the sake of 
completeness. 

\begin{corollary} \label{cor-nc}
Let $G$ be a discrete group and let $H$ be a normal subgroup of finite index in $G$ 
such that $Z(H)$ is finitely generated. Let $T\in\Aut(G)$ be such that $T\in\NC$ and $T|_H=\Id$.  
Then there exists $m\in\N$ such that $G'=\{g\in G\mid T^m(g)=g\}$ is a subgroup of finite index in $G$, $H\subset G'$, 
$T^m$ acts trivially on $G/Z(H)$, and the following hold: For every nontrivial element $x\in G$, 
$G_x\cap G'\ne\{e\}$. Moreover, if $x\not\in G'$, then $\{T^n(x)\}_{n\in\N}$ is infinite $($unbounded$)$. 
\end{corollary}

\begin{proof} If $G$ is finite, then $T^m=\Id$ for some $m\in\N$ and the assertions follow trivially for $G'=G$. 
Suppose $G$ is infinite. Since $G/H$ is finite, there exists $k\in \N$ such that $T^k$ acts trivially on
$H$ and on $G/H$. By Lemma~\ref{zgh}, $T^k$ acts trivially on $G/Z(H)$. Let $F$ denote the unique maximal finite subgroup of 
$Z(H)$ and let $d$ be the order of $F$. Let
$m=kd$ and let $G'=\{g\in G\mid T^m(g)=g\}$. Since $T|_H=\Id$, we have that $H\subset G'$ and $G'$ is a subgroup of finite index in $G$. 
Also, $T^m$ acts trivially on $G/Z(H)$. For every element $x$ of infinite order, $G_x\cap G'\ne\{e\}$. Since $T\in\NC$, we have that 
$T^k\in\NC$. Applying Lemma~\ref{NC} for $T^k$ instead of $T$ and $H$ as above, we get that for every nontrivial  torsion element $x$, 
there exists $l\in\N$ which depends on $x$, such that $T^m(x^l)=x^l\ne\{e\}$, i.e. $G_x\cap G'\ne\{e\}$. 

Now suppose $x\in G$. Then $T^k(x)= xy$ for some $y\in Z(H)$ and $T^{kn}(x)=xy^n$, $n\in\N$. Suppose $\{T^n(x)\}_{n\in\N}$ is finite. 
Then so is $\{T^{kn}(x)\}_{n\in\N}$, hence $y$ has finite order and $y\in F$. Therefore, $T^m(x)=xy^d=x$, and hence 
$x\in G'$. Therefore, if $x\in G\setminus G'$, then $\{T^n(x)\}_{n\in\N}$ is infinite. 
\end{proof}

The following lemma about distality will be useful for proving Theorems \ref{poly-d} and \ref{lattice-distal}.

\begin{lemma} \label{distal-subc} Let $G$ be a discrete group, $T\in\Aut(G)$ and let $H$ be a normal subgroup of finite index in $G$ 
such that $T|_H=\Id$, $Z(H)$ is finitely generated, and $\Sub^c_H$ is closed. Then $\Sub^c_G$ is closed and the following holds:
$T$ acts distally on $\Sub^c_G$ if and only if $T^m=\Id$, for some $m\in\N$. 
\end{lemma}

\begin{proof} Since $\Sub^c_H$ is closed and $G/H$ is finite, it follows from Lemma~\ref{cyc-sg} that $\Sub^c_G$ is closed. 
If $T^m=\Id$, then it acts distally on $\Sub_G$, and hence so does $T$. 
Now suppose $T$ acts distally on $\Sub^c_G$. Then $T\in\NC$ and by Corollary~\ref{cor-nc}, there exists $m\in\N$ such that
$M=\{g\in G\mid T^m(g)=g\}$ has finite index in $G$, $H\subset M$, $T^m$ acts trivially on $G/Z(H)$ and the following hold: 
For every nontrivial element $x\in G$, $G_x\cap M\ne\{e\}$, and if $x\not\in M$, then $\{T^n(x)\}_{n\in\N}$ is infinite. 

We want to show that $T^m=\Id$. 
If possible, suppose $x\in G$ is such that $T^m(x)\ne x$. Then $x\not\in M$. Let $l\in\N$ be the smallest integer such that 
$x^l\in M$.  Then $x^l\ne e$ and, from our assumption, $l\ne 1$ and $l$ is less than or equal to the index of $M$ in $G$. 

 Let $T_1=T^m$. We know from above that $T_1(x)\in x Z(H)$ for all $x\in G$ and 
 $T_1(x)\ne x$. Therefore, $G_x\ne G_x\cap M\ne\{e\}$.  
There exists an unbounded monotone sequence $\{j_k\}\subset\N$ such that 
$T_1^{j_k}(G_x)\to L$ (say) in $\Sub^c_G$. Then $G_x\cap M\subset L\cap M\subset L$. 
Note that $G_x\cap M$ is $T_1$-invariant. Moreover, if $g\in L\cap M$, then $g\in T_1^n(G_x)$ 
for some $n=j_k$, and hence $g\in G_x$ as $T_1(g)=g$. That is, $L\cap M=G_x\cap M$ and it is cyclic. 
As $T_1=T^m$ acts distally on $\Sub^c_G$, we have that $L\cap M\ne L$. Let $a\in L$ be such that $a\not\in M$. 
Since $G$ is discrete, replacing $a$ by $a^{-1}$ if necessary, we get that there exists a sequence $\{n_k\}\subset\N$ 
such that $T_1^{j_k}(x^{n_k})=a$ 
for all large $k$. Since $x^l\in M$, we have that $T_1^{j_k}(x^{ln_k})=x^{ln_k}=a^l$
 for large $k$. Therefore, we have that either $\{n_k\}$ is a constant sequence or $x$ has finite order. 
 In either case, passing to a subsequence if necessary, 
 we can choose $n_k=n_0$ for all $k$. 
 
 Now we have that $a=T_1^{j_k}(x^{n_0})$ for all $k$. Note that $n_0=il+i_0$ for some $i\in\{0\}\cup\N$ and 
 for some fixed $i_0$ with $0\leq i_0 <l$. Let $k\in\N$ be fixed. Then $a=T_1^{j_k}(x^{i_0})x^{il}\in x^{i_0}M$ 
 as $x^{il}\in M$ and $T_1(x)\in xM$. Here, $i_0\ne 0$ as $a\not\in M$. Since $T_1$ acts trivially on $G/Z(H)$, 
 $T_1(x^{i_0})=x^{i_0}y$ for some $y\in Z(H)$. Therefore, 
$$a=T_1^{j_k}(x^{n_0})=x^{il}T_1^{j_k}(x^{i_0})=x^{il}x^{i_0}y^{j_k}=x^{n_0}y^{j_k}.$$ 
Hence $y^{j_k}=x^{-n_0}a$ for all $k\in\N$. Since $\{j_k\}$ is unbounded, we get that $y$ has 
finite order, and hence $y\in F$, where $F$ is the unique maximal finite subgroup of $Z(H)$ with order $d$ (say). 
This implies that $T^{md}(x^{i_0})=T_1^d(x^{i_0})=x^{i_0}y^d=x^{i_0}$ and that 
$\{T^n(x^{i_0})\}_{n\in\N}$ is finite. This leads to a contradiction as 
$1<i_0<l$ and $x^{i_0}\not\in M$. Therefore, 
$T^m(x)=x$ for all $x\in G$, and hence $T^m=\Id$.
\end{proof}

Every polycyclic group contains a unique maximal nilpotent normal subgroup.  The following theorem about distality for 
polycyclic groups generalises Theorem 3.10 of \cite{PaS20} since lattices in a connected solvable Lie group are
polycyclic.  Theorem~\ref{poly-d} holds in particular for any discrete solvable subgroup of a connected Lie group due to 
Lemma~\ref{solvable-Lie}, and it will be useful in proving Theorem~\ref{lattice-distal} for lattices in a connected Lie group. 
Note that the class of polycyclic groups is strictly larger than that of lattices in connected solvable Lie groups 
(see Examples 4.29--4.33 in \cite{Ra72}). Example 3.11 in \cite{PaS20} illustrates that not all the statements in 
the theorem are equivalent. 

\begin{theorem} \label{poly-d} Let $G$ be a discrete polycyclic group and let $T\in\Aut(G)$. Let $G_\nil$ be the unique 
maximal nilpotent normal subgroup of $G$. Then $\Sub^c_G$ is closed and $(1-2)$ are equivalent as well as $(3-6)$ are equivalent. 
\begin{enumerate}
\item[{$(1)$}] $T\in\NC$.
\item[{$(2)$}] There exist $n\in\N$ and a subgroup $G'$ of finite index in $G$ containing $G_\nil$ such that 
$T^n|_{G'}=\Id$, the identity map on $G'$ and $G'\cap G_x\ne\{e\}$ for every nontrivial element $x$ in $G$. 
\end{enumerate}
\begin{enumerate}
\item[{$(3)$}] $T$ acts distally on $\Sub^c_G$.
\item [{$(4)$}] $T$ acts distally on $\Sub^a_G$.
\item [{$(5)$}] $T$ acts distally on $\Sub_G$.
\item[{$(6)$}] $T^n=\Id$ for some $n\in\N$.
\end{enumerate}
If $G$ is nilpotent, then $(1-6)$ are equivalent.
\end{theorem}

\begin{proof} Since $G$ is polycyclic, every subgroup of $G$ is finitely generated, by Lemma 3.3 of \cite{PaS20}, 
$\Sub^c_G$ is closed. If $G$ is finite, then so is $\Sub_G$ and $(1-6)$ hold trivially as $T^n=\Id$ for some $n\in\N$. Now suppose 
$G$ is infinite. Suppose (1) holds. We know that $G_\nil$ is characteristic in $G$, and hence it is $T$-invariant. Since 
$G_\nil$ is finitely generated and nilpotent, the set $R_g$ of roots of $g\in G$ is finite for every $g\in G$ 
(cf.\ \cite{He77}, Theorems 3.1.13 and 3.1.17). As $T|_{G_\nil}\in\NC$, by Proposition 3.8 of \cite{PaS20}, $T^{k_1}$ acts trivially on 
$G_\nil$ for some $k_1\in\N$. If $G=G_\nil$, then (2) holds for $G'=G$ and $n=k_1$. 

Now suppose $G\ne G_\nil$. Suppose $G/G_\nil$ is finite. Since $T^{k_1}\in\NC$ and $Z(G_\nil)$ is finitely generated, 
we get from Corollary~\ref{cor-nc} that (2) holds for some subgroup $G'$ containing $G_\nil$ and some $n\in\N$. Now suppose 
$G/G_\nil$ is infinite. 
Let $\bar T: G/G_\nil\to G/G_\nil$ be the natural projection. Since $G$ is infinite and polycyclic, it admits a strongly polycyclic subgroup (say)
$L$ of finite index. Passing to a subgroup of finite index if necessary, we may assume that $L$ is $T$-invariant and normal in $G$ and that 
$L/L_\nil$ is abelian and torsion-free (cf.\  \cite{Ra72}, Corollary 4.11). 
Since $L_\nil$ is nilpotent and characteristic in $L$, it is normal in $G$ and it is contained in $G_\nil$. 

Here, $L/L_\nil$ is infinite, as $G/G_\nil$ is so. 
Let $\bar T:L/L_\nil\to L/L_\nil$ be the natural projection. Since $L/L_\nil$ is torsion-free, by Lemma 3.5 of \cite{PaS20}, 
$\bar T|_{(L/L_\nil)}\in\NC$ and hence $\bar T^{k_2}$ acts trivially on $L/L_\nil$ for some $k_2\in\N$. 
Let $k={\rm lcm}(k_1,k_2)$ and let $T_1=T^k$. Then $T_1$ acts trivially on $G_\nil$, $L_\nil$ and on $L/L_\nil$.
Then $T_1|_{L_\nil}=\Id$ and $T_1(x)\in xL_\nil$ for all $x\in L$. If possible, suppose there exists $x\in L$, such that 
$T_1(x^i)\ne x^i$ for all $i\in\N$. Since $L_\nil$ is torsion-free, by Lemma 3.12 of \cite{SY19}, we get that $T_1|_L\not\in\NC$. 
Since $T|_L\in\NC$, this leads to a contradiction. Hence, given any $x\in L$, 
$T_1(x^i)=x^i$ for some $i\in\N$ (which depends on $x$). Since $L$ is finitely generated and $L/L_\nil$ is abelian, we get a 
subgroup (say) $L'$ of finite index in $L$ such that $T_1$ acts trivially on $L'$. Let $G_1=\{g\in G\mid T^k(g)=g\}$. Then 
$L'G_\nil\subset G_1$ and $G_1$ is a subgroup of finite index in $G$. 
Let $H$ be a normal ($T$-invariant) subgroup of finite index in $G_1$ such that $G_\nil\subset H\subset G_1$. 
Since $G$ is polycyclic, so is $Z(H)$, and hence it is finitely generated. Since $T^{k}|_{H}=\Id$ and 
$T^k\in\NC$, applying Corollary~\ref{cor-nc} for $T^k$ instead of $T$, we get that (2) holds for some subgroup 
$G'$ such that $G_\nil\subset H\subset G'$ and some $n\in\N$. 

Now suppose (2) holds. If $G$ is nilpotent, then $G=G_\nil=G'$ and hence $(2)\implies (1)$. Suppose $G$ is not nilpotent. 
Let $G'$ and $n$ be as in (2). Then $G'$ is a subgroup of finite index (say) $m$ and $T^n|_{G'}=\Id$. We show that $T^n\in\NC$. 
This would imply that $T\in\NC$. If $x\in G$ has infinite order, then $T^n(G_{x^m})=G_{x^m}\ne\{e\}$.
If a nontrivial $x\in G$ has finite order, then there exists $l\in\N$ such that $e\ne x^l\in G'$ and $T^n(G_{x^l})=G_{x^l}\ne\{e\}$. 
Therefore, in either case, $T^{nm_j}(G_x)\not\to\{e\}$ for any sequence $\{m_j\}\subset\Z$. Hence, $T^n\in\NC$ and (1) holds.

We know that $(6)\implies (5)\implies (4)\implies (3)$. Now we show that $(3)\implies (6)$. Suppose (3) holds. That is, 
$T$ acts distally on $\Sub^c_G$. Then $T\in\NC$ and hence (2) holds.  
Let $G'$ be a subgroup of $G$ satisfying the first condition in (2). Here $G'$ is a subgroup of finite index in $G$ such that 
$T^n|_{G'}=\Id$ for some $n\in\N$. We may replace $G'$ by a normal subgroup of finite index and assume that 
it is normal in $G$. As $Z(G')$, being polycyclic, is finitely generated, by Lemma~\ref{distal-subc}, we get that $T^n=\Id$ for 
some $n\in\N$, and hence (6) holds. Therefore, $(3-6)$ are equivalent. 

If $G$ is nilpotent, then $G=G_\nil$ and $2\Leftrightarrow 6$, hence $(1-6)$ are equivalent. 
\end{proof}

For a connected Lie group $G$, recall that $\Aut(G)$ is identified with a subgroup of $\GL(\Gc)$ and the topology inherited by it
as a subspace of $\GL(\Gc)$ coincides with the compact-open topology and it is a Lie group. The following theorem generalises 
Corollary 3.9, Theorem 3.10 and Theorem 3.16 of \cite{PaS20} which are for lattices in simply 
connected nilpotent, simply connected solvable and connected semisimple Lie groups respectively. We will illustrate by constructing 
several counter examples that the theorem is the best possible result in this direction. Note that for any lattice $\Ga$ in a connected 
Lie group, $\Sub^c_\Ga$ is closed in $\Sub_\Ga$ by Corollary~\ref{cyc-cl}. 

\begin{theorem} \label{lattice-distal} Let $\Gamma$ be a lattice in connected Lie group $G$. Let $T\in\Aut(\Ga)$. 
Then $(1-2)$ are equivalent and $(3-6)$ are equivalent.

\begin{enumerate} 
\item[{$(1)$}] $T\in\NC$.
\item[{$(2)$}] There exist $n\in\N$ and a subgroup $\Gamma'$ of finite index in $\Ga$ such that $\Ga_\nil\subset\Ga'$,  
$T^n|_{\Ga'}=\Id$, the identity map on $\Ga'$ and for every nontrivial element $x\in \Ga$, $G_x\cap\Ga'\ne\{e\}$. 
\end{enumerate}
\begin{enumerate}
\item[{$(3)$}] $T$ acts distally on $\Sub^c_\Ga$.
\item [{$(4)$}] $T$ acts distally on $\Sub^a_\Ga$.
\item [{$(5)$}] $T$ acts distally on $\Sub_\Ga$.
\item[{$(6)$}] $T^n=\Id$ for some $n\in\N$.
\end{enumerate}
If $G$ or $\Ga$ is nilpotent then $(1-6)$ are equivalent. \\

\noindent Suppose $\tau\in\Aut(G)$ is such that $\tau$ keeps $\Gamma$ invariant and $T=\tau|_\Ga$. Suppose the radical of $G$ is simply 
connected and nilpotent. If the maximal compact connected normal subgroup of a Levi subgroup of $G$ is normal in $G$, then  
$(1-6)$ are equivalent and they are equivalent to each of the following statements:
\begin{enumerate}
\item [{$(7)$}] $\tau\in\NC$. 
\item[{$(8)$}] $\tau$ acts distally on $\Sub^a_G$.
\item[{$(9)$}] $\tau$ acts distally on $\Sub_G$.
\item[{$(10)$}] $\tau$ is contained in a compact subgroup of $\Aut(G)$.
\end{enumerate}
Moreover, if a Levi subgroup of $G$ has no compact factors, then $(1-10)$ are equivalent to the following: 
\begin{enumerate}
\item[{$(11)$}] $\tau^n=\Id$ for some $n\in\N$. 
\end{enumerate}
\end{theorem}

\begin{proof}  
If $G$ is simply connected and nilpotent (resp.\ connected and semisimple), then the assertions in the theorem follow 
from Corollary 3.9 (resp.\ Theorem 3.16) of \cite{PaS20}. 

Suppose $G$ is compact. Equivalently, $\Ga$ is finite, hence $T^n=\Id$ for some $n\in\N$, and $(1-6)$ are 
equivalent. Moreover, when the radical of $G$ is simply connected (and nilpotent), 
the compact group $G$ is semisimple, and hence $\Aut(G)$ is also compact. Therefore, $(1-10)$ are equivalent by 
Theorem 3.16 of \cite{PaS20}, and the additional condition that a Levi subgroup of $G$ has no compact factor implies 
that $G$ is trivial in this case and hence $\tau=\Id$ and $(1-11)$ are equivalent. 

We now assume that $G$ is not compact (equivalently, $\Ga$ is not finite). 
If $G$, or more generally, $\Ga$ is nilpotent, then $\Ga=\Ga_\nil$ and $(2)\implies (6)$, and hence $(2-6)$ are equivalent, 
and $(1-6)$ are equivalent by Theorem~\ref{poly-d} as $\Ga$ is polycyclic in this case. 

If $\Ga$ is solvable, then by Lemma~\ref{solvable-Lie}\,(2), $\Ga$ is polycyclic, 
and it follows from Theorem~\ref{poly-d} that $(1-2)$ (as well as $(3-6)$) are equivalent.  Suppose $\Ga\ne \Ga_\rad$. 
Note that $(2)\implies (1)$ follows easily as in the proof of Theorem~\ref{poly-d}\,($(2)\implies (1)$).  
Now we show that $(1)\implies (2)$.

Suppose (1) holds. We know that $T(\Ga_\rad)=\Ga_\rad$ and that $\Ga_\rad$ is polycyclic by Proposition~\ref{lattice-rad}. Note that 
$\Ga_\nil\subset \Ga_\rad$ and it is also the unique maximal nilpotent normal subgroup of $\Ga_\rad$. By Theorem~\ref{poly-d}, 
$\Ga_\rad$ has a subgroup (say) $\Delta$ of finite index such that $\Ga_\nil\subset\Delta$ and $T^{n_1}|_{\Delta}=\Id$. 
Passing to a normal subgroup of finite index of $\Delta$ if necessary, we get that $\Delta$ is normal in $\Ga$. 
Suppose $\Ga_\rad$ is a subgroup of finite index in $\Ga$. As $T^{n_1}\in\NC$ and $Z(\Delta)$ is finitely 
generated, by Corollary~\ref{cor-nc}, we get that $(2)$ holds for some $\Ga'$ containing $\Delta$, and some $n\in\N$. 

Now suppose $\Ga/\Ga_\rad$ is infinite. By Proposition~\ref{quo-poly}, $\Ga$ has a normal subgroup (say) $\Lam$ of 
finite index such that $\Ga_\rad=\Lam_\rad$, $\Lam/\Ga_\rad$ is torsion-free and the set $R_g$ of roots of $g$ in 
$\Lam/\Ga_\rad$ is finite for every $g\in \Lam/\Ga_\rad$. 
Passing to a subgroup of finite index if necessary, we may also assume that $T(\Lam)=\Lam$. 

Suppose $T\in\NC$. Then $T|_\Lam\in\NC$. 
Let $\eta:\Ga\to\Ga/\Ga_\rad$ be the natural projection and let $\bar T$ be the
automorphism of $\Ga/\Ga_\rad$ corresponding to $T$. Since $\Ga$ is discrete and $\eta(\Lam)$ is torsion-free, we get 
by Lemma 3.5 of \cite{PaS20} that $\bar T|_{\eta(\Lam)}\in\NC$. Since the set $R_g$ of roots of $g$ in $\eta(\Lam)$ is finite for every
$g\in \eta(\Lam)$, we get from Proposition 3.8 of \cite{PaS20} that $\bar T^{n_2}|_{\eta(\Lam)}=\Id$ for some $n_2\in\N$. Therefore,
$T^{n_2}$ acts trivially on $\Lam/\Ga_\rad$. 

Suppose $\Ga_\rad$ is a  finite group of order $n'$ (say). Replacing $n_2$ by its multiple in $\N$ if necessary, we may assume that 
$T^{n_2}$ acts trivially on $\Lam/\Ga_\rad$ as well as on $\Ga_\rad$. Now we may again replace $n_2$ by $n_2n'$ and get that 
$T^{n_2}|_\Lam=\Id$. By Lemma~\ref{solvable-Lie}, $Z(\Lam)$ is finitely generated, and since $T^{n_2}\in\NC$, by Corollary~\ref{cor-nc}, 
we get that (2) holds for some subgroup $\Ga'$ containing $\Lam$, and some $n\in\N$. We note here that 
$\Ga_\nil\subset\Ga_\rad\subset \Lam\subset\Ga'$. 

Now suppose $\Ga_\rad$ is infinite. Replacing  $n_2$ by its multiple in $\N$ if necessary, we may assume  that $T^{n_2}$ acts trivially on 
$\Ga/\Lam$ and $\Lam/\Ga_\rad$. We also have from above that $\Ga_\rad$ has a subgroup 
$\Delta$ of finite index such that $T^{n_1}|_{\Delta}=\Id$, where $n_1\in\N$ and $\Ga_\nil\subset\Delta$. Here, 
$\Delta$ is infinite since $\Ga_\rad$ is so. Passing to a subgroup of finite index if necessary, we may assume 
without loss of generality that $\Delta$ is strongly polycyclic, and hence torsion-free. Let $m$ be the index of $\Delta$ in $\Ga_\rad$ 
and let $L_m$ be the group generated by 
$\{x^m\mid x\in\Ga_\rad\}$. Then $L_m\subset \Delta$ and it is characteristic in $\Ga$. Since $\Ga_\rad$ is polycyclic, $L_m$ is a 
subgroup of finite index in $\Ga_\rad$ (cf.\ \cite{Ra72}, Lemma 4.4). 
Replacing $\Delta$ by $L_m$, we may assume that $\Delta$ is $T$-invariant and normal in $\Ga$. 
Replacing $n_2$ by a larger multiple of it in $\N$ if necessary, we may assume that $T^{n_2}$ acts trivially on 
$\Ga_{\rad}/\Delta$. Let $n_3=\lcm(n_1,n_2)$.
Then $T^{n_3}$ acts trivially on $\Ga/\Lam$, $\Lam/\Ga_\rad$, $\Ga_\rad/\Delta$ and also on $\Delta$. Now for any $x\in\Lam$, 
$T^{n_3m}(x)\in x\Delta$ where $m$ as above is the index of $\Delta$ in $\Ga_\rad$. 
Then $T^{n_3m}$ acts trivially on both $\Delta$ and on $\Lam/\Delta$. By Lemma~\ref{zgh}, $T^{n_3m}$ acts trivially on $\Lam/Z(\Delta)$. 
If $Z(\Delta)=\{e\}$. Then $T^{n_3m}$ acts trivially on $\Lam$, and since $T^{n_3m}\in\NC$, by Corollary~\ref{cor-nc}, 
we get that $(2)$ holds for some subgroup $\Ga'$ containing $\Lam$, and some $n\in\N$. Now $\Ga_\nil$ may not be contained in 
$\Delta$, but $\Ga_\nil\subset\Ga_\rad\subset\Lam\subset\Ga'$.

Now suppose $Z(\Delta)\ne\{e\}$. 
Note that $Z(\Delta)$, being abelian and strongly polycyclic, is isomorphic to $\Z^d$ for some $d\in\N$. 
Note also that $Z(\Delta)$ is characteristic in 
$\Delta$ and hence normal in $\Ga$. Therefore, we have a natural homomorphism $\varrho:\Ga\to \GL(d,\Z)$, 
$\varrho(x)=\inn(x)|_{Z(\Delta)}$, $x\in\Ga$, where $\inn(x)$ is the inner automorphism by $x$ in $\Ga$. 
Note that $\ker\varrho$ is a normal subgroup of $\Ga$
which contains $\Delta$. Since $\varrho(\Ga)$ is finitely generated, by Corollary 17.7 of \cite{Bo19}, $\varrho(\Ga)$ has a subgroup 
(say) $\Ga''$ of finite index which is net, i.e.\ for every $g\in\Ga''$, the multiplicative group generated by eigenvalues of $g$ in 
$\Cm\setminus\{0\}$ is torsion-free. 
We may replace $\Ga''$ by a subgroup of finite index if necessary, and assume that it is normal in $\varrho(\Ga)$. 
Now $\varrho^{-1}(\Ga'')\cap\Lam$ is a normal subgroup of finite index in $\Gamma$ and it contains $\Delta$. Replacing $\Lam$ by 
$\varrho^{-1}(\Ga'')\cap \Lam$, we may assume that $\varrho(\Lam)$ is net. As $\Lam$ is a normal subgroup of finite index in $\Ga$, 
$\Lam\cap \Ga_\rad=\Lam_\rad$ is normal in $\Ga$ and it is a subgroup of finite index in $\Ga_\rad$. 

Let $T_1=T^{n_3m}$. As noted above $T_1$ acts trivially on $\Delta$ and $\Lam/Z(\Delta)$ where $m$ as above is the index of 
$\Delta$ in $\Ga_\rad$. We show that $T_1$ acts trivially on $\Lam$. Since $T\in\NC$, we have that $T_1\in\NC$. 
Let $x\in\Lam$ be fixed. Here, $T_1(x)=xy$ for some $y\in Z(\Delta)$. If $y=e$, then $T_1(x)=x$. If possible, suppose $y\ne e$. 
If $xy=yx$, then $T_1(x^n)=x^ny^n\ne x^n$ for every $n\in\N$, as $Z(\Delta)$ is torsion-free. By Lemma 3.12 of \cite{SY19}, it leads to
to a contraction as $T_1\in\NC$. Now suppose $xy\ne yx$. By Lemma 3.12 of \cite{SY19}, we get that 
$T_1(x^l)=x^l$, for some $l\in\N\setminus\{1\}$ which depends on $x$. This implies that $(xy)^l=x^l$, and hence that 
$x^{l}yx^{-l}=y$. In particular, $\varrho(x)$ has an eigenvalue which is a nontrivial root of unity. 
This leads to a contradiction as $x\in\Lam$ and $\varrho(\Lam)$ is net. 
Therefore, $T_1(x)=x$ for all $x\in\Lam$. 

As $T^{n_1}$ acts trivially on $\Ga_\nil$ and $n_1$ is a factor of $n_3$, we have that $T_1=T^{n_3m}$ acts trivially on $\Ga_\nil$. 
Replacing $\Lam$ by $\Lam\Ga_\nil$, we get $\Ga_\nil\subset\Lam$, which is a subgroup of finite index in $\Ga$.
As $Z(\Lam)$ is finitely generated and $T_1\in\NC$, we get from Corollary~\ref{cor-nc} that (2) holds for some subgroup $\Ga'$ containing 
$\Lam$, and some $n\in\N$. Therefore, $(1)\implies (2)$.

We know that $(6)\implies (5)\implies (4)\implies (3)$. Suppose (3) holds. Note that $(3)\implies (1)\implies (2)$. Let $\Ga'$ be a subgroup 
of finite index in $\Ga$ such that $T^n|_{\Ga'}=\Id$ for some $n\in\N$. We may replace $\Ga'$ by a subgroup of finite index and assume that
$\Ga'$ is normal in $\Ga$. Since $Z(\Ga')$ is finitely generated, we get from Lemma~\ref{distal-subc} that $T^n=\Id$ for some $n\in\N$.
Therefore, $(3-6)$ are equivalent. 

Let $\tau\in\Aut(G)$ be as in the hypothesis such that $\tau|_\Ga=T$. 
If $G$ is nilpotent, then it is simply connected and $(1-11)$ are equivalent by Corollary 3.9 of \cite{PaS20}. If $G$ is a
 connected semisimple Lie group, then by Theorem 3.16 of \cite{PaS20}, we get that $(1-10)$ are equivalent and if $G$ has no 
compact factors, then $(1-11)$ are equivalent. Now we suppose that $G$ is neither nilpotent nor semisimple. 

Now suppose $G$ is a Lie group whose radical is simply connected and nilpotent. Then $G$ has no nontrivial compact connected 
central subgroup, and by Theorem 4.1 of \cite{SY19}, we have that $(7-10)$ are equivalent. Also, $(8)\implies (4)\implies (6)\implies (2)$ is
obvious. Now suppose $(2)$ holds. 
Suppose $T^n|_{\Ga'}=\Id$ for some $n\in\N$ and a subgroup $\Ga'$ of finite index in $\Ga$ as in (2). 
We replace $T$ by $T^n$ (and $\tau$ by $\tau^n)$ for simplicity and assume the $\tau|_{\Ga'}=T|_{\Ga'}=\Id$, where 
$\Ga'$ is also a lattice in $G$. 

Let $K$ be the maximal compact connected normal subgroup of $G$ and let $N$ be the nilradical of $G$. By our assumption that the 
radical is simply connected and nilpotent, we have that $N$ is simply connected and $K\cap N=\{e\}$. We also have that 
$G=SN$, a Levi decomposition, where $S$ is a (semisimple) Levi subgroup of $G$. From our earlier assumption that 
$G$ is neither semisimple nor nilpotent, we have that both $S$ and $N$ are nontrivial. 
Now suppose the maximal compact connected normal subgroup of $S$ is normal in $G$. 

Suppose $K$ is trivial. Then from our assumption on $S$, it follows that $S$ has no nontrivial compact connected normal 
subgroup; equivalently,  it has no compact factors. By Corollary 8.28 of \cite{Ra72}, $\Ga'\cap N$ is a lattice in 
$N$ and $T$ acts trivially on it. By Corollary 3.9 of \cite{PaS20}, we get that $\tau$ acts trivially on
$N$. Therefore, it acts trivially on $\Ga' N$. Let $\pi:G\to G/N$ be the natural projection. By Corollary 8.28 of \cite{Ra72}, 
we have that $\pi(\Ga')$ is a lattice in $\pi(G)$. As $\pi(S)$ has no compact factors, by Theorem 3.16 of \cite{PaS20}, 
$\tau^n$ acts trivially on $\pi(G)=\pi(S)$ for some $n\in\N$. By Lemma~\ref{zgh}, $\tau^n$ acts trivially on $G/Z(N)$. We may again 
replace $\tau$ by $\tau^n$ and assume that $\tau$ acts trivially on $\Ga'N$ as well as on $G/Z(N)$. 

Note that $\tau(S)$ is also a Levi subgroup of $G$, and hence $\tau(S)=aSa^{-1}$ for some $a\in N$. Therefore, $\inn(a)^{-1}\circ\tau$ is
an automorphism of $S$. As $S$ is semisimple, there exists $m\in\N$ such that $(\inn(a^{-1})\circ\tau)^m|_S$ is an inner automorphism 
of $S$. As $\inn(N)$ is normal in $\Aut(G)$, there exist $x\in N$ and $y\in S$, $\tau^m(s)=xysy^{-1}x^{-1}$ for all $s\in S$. We first show that
$y\in Z(S)$. Without loss of any generality, we may replace $\tau$ by $\tau^m$ and assume that $\tau(s)=xysy^{-1}x^{-1}$ for all $s\in S$. 
As $\tau$ acts trivially on $G/Z(N)$, we get that for $s\in S$, $\tau(s)=xysy^{-1}x^{-1}=sz_s$ for some $z_s\in Z(N)$ which depends on $s$. 
This implies that $s^{-1}ysy^{-1}=s^{-1}x^{-1}sz_sx=s^{-1}x^{-1}sxz_s$, $s\in S$. Therefore, $s^{-1}ysy^{-1}\in S\cap N$ for all $s\in S$. 
Since $S$ is connected and $S\cap N$ is a discrete subgroup of $S$, we get that $sy=ys$, $s\in S$, and hence 
$\tau|_S=\inn(x)|_S$ for some $x\in N$. 

Now $\tau(s)=xsx^{-1}=sz_s\in s\,Z(N)$, $s\in S$. We show that $x\in Z_G(S)$, which would imply that 
$\tau|_S=\Id$. If $x=e$, then this is obvious. Suppose $x\ne e$. For $g\in \Ga'$, if $g=g_sg_n$, for some $g_s\in S$ and $g_n\in N$, then 
$g_s\in \Ga'N$ and hence $g_s=\tau(g_s)=xg_sx^{-1}$. 

Observe that $S\cap N$ is a discrete central subgroup of $S$, and hence it contains a subgroup of finite index which is central in $G$. 
Since $N/Z(N)$ is simply connected, we get that $S\cap N\subset Z(N)$, and hence $S\cap N\subset Z(G)$. We know that $G/N$
is isomorphic to $S/(S\cap N)$ which is a connected semisimple Lie group without compact factors. Moreover, $\Ga N/N$ is a lattice in
$G/N$.  Let $\xi$ denote the canonical isomorphism from $G/N$ to $S/(S\cap N)$, and let 
$\pi'=\xi\circ\pi$, where $\pi:G\to G/N$ is as above. If $g=g_sg_n$, for some $g_s\in S$ and $g_n\in N$, then 
$\pi(g)=\pi(g_s)=g_sN$ and $\pi\circ\xi(g)=\xi(g_sN)=g_s(S\cap N)$. As $\pi(\Ga)$ is a lattice in $G/N$, we have that 
$\pi'(\Ga)$ is a lattice in $\pi'(S)$, where 
$\pi'(\Ga)=\{\ga_s(S\cap N)\mid \ga_s\in S\hbox{ and }\ga_s\ga_n\in\Ga\hbox{ for some }\ga_n\in N\}$. 
Now $\pi'(\Ga')$, being a subgroup of finite index in $\pi'(\Ga)$, is also a lattice in $\pi'(S)$.

Let $\NN$ be the Lie algebra of $N$ contained in the Lie algebra $\Gc$ of $G$. Since $N$ is simply connected and nilpotent, 
the exponential map restricted $\NN$ is a homeomorphism onto $N$ with $\log: N\to \NN$ as its inverse (see \cite{Va84}). 
We now define an action of $\pi'(S)=S/(S\cap N)$ on $\NN$. 
 Let $\rho:S/(S\cap N)\to \GL(\NN)$ be defined as follows: for $s\in S$, let $\rho(s(S\cap N))=\Ad(s)|_\NN$, $s\in S$. 
Since $S\cap N$ is central in $N$, we get that $\rho$ is well defined on $\pi'(S)$ 
and it is a representation of $\pi'(S)$ which is a connected semisimple Lie group without compact factors. 
Let $v_x=\log x$ for $x$ as above. 
Let $\ga\in\Ga'$. Then $\pi'(\ga)=\pi'(\ga_s)\in \pi'(\Ga')$ for some $\ga_s\in S$ and $x\ga_sx^{-1}=\ga_s$. Now 
 $$
 \rho(\ga_s(S\cap N))v_x=\Ad(\ga_s)v_x=\log(\ga_sx\ga_s^{-1})=\log(x)=v_x.$$
 Since $\pi'(\Ga')$ is a lattice in $\pi'(S)$, by the Borel density theorem, we get that $\rho(\pi'(s))(v_x)=v_x$ for all $s\in S$. This 
implies that $sxs^{-1}= x$ for all $s\in S$, and hence that $\tau|_S=\Id$. Since $\tau|_N=\Id$ and $G=SN$, we get that $\tau=\Id$. 
Therefore, $(11)$ holds. That is, $(2)\implies (11)$ in the case at hand. 
Since $(11)\implies (10)\implies (6)\implies (1)$, and $(1-2)$, $(3-6)$ as 
well as $(7-10)$ are equivalent, we have that $(1-11)$ are equivalent if a Levi subgroup of $G$ has no 
compact factors and the radical of $G$ is simply connected and nilpotent.

Now suppose $K$ is nontrivial. It is semisimple as $K\cap N=\{e\}$. From our assumption on $S$, 
it follows that $K$ is the product of all compact (simple) factors of the Levi subgroup $S$. Then either $S=K$ or 
$S=KS'=S'K$ where $S'$ is a product of all non-compact simple factors of $S$. Moreover $G=KN$ or $G=KS'N$, 
where $N$ is nontrivial. It is enough to show that $(10)$ holds. 

Let $\psi:G\to G/K$ be the natural projection. Since $\tau(K)=K$, we have the corresponding action of $\tau$ on $G/K$. 
Then $\psi(G)$ has no nontrivial compact connected normal subgroup, $\psi(\Ga')$ is a lattice in $\psi(G)$ and 
we get from (2) that $\tau^n$ acts trivially on $\psi(\Ga')$. Either $\psi(G)=\psi(N)$ or
$\psi(G)=\psi(S'N)$. Note that $\psi(N)$ is simply connected as $N$ is so. If $G=KN$, then we have by Corollary 3.9 of \cite{PaS20} 
 that $\tau^n$ acts trivially on $\psi(G)=\psi(N)$. If $G=KS'N$, then $\psi(S')$ is a Levi subgroup of $\psi(G)$ and it has no compact 
 factors. From above, we have that $\tau^n$ acts trivially on $\psi(G)$ for some $n\in\N$. Therefore, in either case, 
$\tau^n(x)\in xK$, for all $x\in G$. This implies that, for any $x\in N$, $x^{-1}\tau^n(x)\in K\cap N=\{e\}$. 
That is, $\tau^n$ acts trivially on $N$. As $K\subset S$, it follows that $\tau^n(S)=S$. Suppose $S'$ is nontrivial. Then $\tau^n(S')=S'$. 
Since $S'\cap K$ is a finite central subgroup in $S'$ and $S'$ is connected, it follows that for all $x\in S'$, $x^{-1}\tau^n(x)$ is 
contained in the connected component of the identity in $S'\cap K$ which is trivial. Therefore, $\tau^n(x)=x$ for all $x\in S'$, and 
hence $\tau^n$ acts trivially on $S'N$ which is a co-compact normal subgroup of $G$. 

Since $K$ is a compact connected semisimple Lie group, its automorphism group contains the group of inner automorphisms of 
$K$ as a subgroup of finite index, and hence it is compact. Moreover, as $K$ is normal and $N$ is simply connected, elements of $K$ 
centralise $N$, and also $S'N$ if $S'$ is nontrivial. Note that $G=KN$ (resp.\ $G=KS'N$) and 
$\tau$ acts trivially on $N$ (resp.\ $S'N$). Since $\tau(K)=K$ and $\tau^n|_{N}=\Id$ and $\tau^n|_{S'}=\Id$ if $S'$ is nontrivial, 
replacing $n$ by its multiple in $\N$, we have that $\tau^n$ is an inner automorphism of $G$ by an element of $K$. 
Therefore, $\tau$ generates a relatively compact group in $\Aut(G)$. That is, $(2)\implies (10)$, and hence $(1-10)$ are equivalent. 
\end{proof}

Now we illustrate by examples that Theorem~\ref{lattice-distal} is the best possible result for lattices in a connected Lie group. If $G$ is a nontrivial
compact connected semisimple Lie group, then the trivial group is a lattice which is invariant under any automorphism of $G$ and 
$\Aut(G)$ is a nontrivial compact group containing elements of infinite order. Hence for such a $G$, $(10)$ above holds but $(11)$ can not 
hold in general. 

Example 3.11 in \cite{PaS20} shows that a simply connected solvable Lie group $G$ can admit a lattice $\Ga$ and an 
automorphism $\tau\in\Aut(G)$ such that $\tau$ keeps $\Ga$ invariant, $\tau|_\Ga\in\NC$ but $\tau$ does not act distally on 
$\Sub^c_\Ga$. This illustrates that in Theorem~\ref{lattice-distal}, neither $(1-6)$ nor $(1-10)$ are equivalent in general. If 
$G=\TT^d$, $d\geq 2$, then its lattices are finite and any automorphism of $G$ keeps the finite group 
$G_n=\{g\in G\mid g^n=e\}$ invariant for any fixed $n\in\N$. Note that $\Aut(G)$ is isomorphic to $\GL(d,\Z)$ and, 
by Selberg's Lemma, it admits a subgroup of finite index which is torsion-free. Hence for such a $G$, 
(6) above holds for any lattice but (10) or (11) can not hold in general. 

Now we give an example of a class of groups $G=K\ltimes\R^d$, $d\geq 3$, where $K$ is any nontrivial compact connected 
subgroup of $\GL(d,\R)$ and $G$ admits an automorphism $\tau$ and a lattice $\Ga$ such that $\tau|_\Ga=\Id$
but $\tau$ does not generate a relatively compact group in $\Aut(G)$, hence $(1-6)$ of Theorem~\ref{lattice-distal} hold in this case, 
but none of $(7-11)$ holds. The group $G$ as above has compact or trivial Levi subgroups, and it is solvable if 
$K$ is abelian. Example \ref{ex1} together with the examples mentioned above illustrate that the conditions in Theorem~\ref{lattice-distal} 
that the radical is simply connected and nilpotent and the maximal compact connected normal subgroup 
of a Levi subgroup is normal in the whole group are necessary for the equivalence of $(1-10)$.  

\begin{example} \label{ex1}
Let $G=K\ltimes\R^d$, for some $d\geq 3$, and let $K$ be any nontrivial compact connected subgroup of $\GL(d,\R)$, where 
the group operation is given by $(k_1,x_1)(k_2,x_2)=(k_1k_2, k_2^{-1}(x_1)+x_2)$. Let 
$\Ga=\Z^d\subset\R^d$. Since $K$ is compact, $\Ga$ is a lattice in $G$. Choose $k\in K$ and $z\in \R^d$ such that $k(z)\ne z$.
Let $\tau=\inn(z)$, the inner automorphism by $z$, i.e.\ $\tau(g)=zgz^{-1}$ for all $g\in G$. Then $\tau$ acts trivially on $\Ga$. 
Now we prove that the closed subgroup generated by $\tau$ is noncompact in $\Aut(G)$. It is enough to 
show that $\{\tau^n(k)\mid n\in\N\}$ is unbounded. Note that $\tau^n(k)=(k, nk^{-1}(z) - nz)=(k, n(k^{-1}(z) - z))=(k,ny)$, where
$y=k^{-1}(z) - z\in \R^d$. Since $k(z)\ne z$, $y\ne 0$. Therefore, $\{ny\mid n\in\N\}$ is unbounded, and hence
$\{\tau^n(k)\mid n\in\N\}$ is unbounded. Note that $K$ could be chosen to be abelian or semisimple and it is not normal in $G$.
Here $(1-6)$ of Theorem~\ref{lattice-distal} hold but none of the $(7-11)$ of the same theorem holds. 
\end{example}

\section{Expansive actions of automorphisms on Sub$_{\mathbf \Ga}$ of lattices ${\mathbf \Ga}$ in Lie groups}

In this section, we study expansive actions of automorphisms of $G$ on $\Sub^c_G$ for a certain class of discrete groups $G$ 
which include discrete polycyclic groups. We also show that a lattice $\Ga$ in a connected noncompact Lie group does not 
admit any automorphism which acts expansively on $\Sub^c_\Ga$. 

  For a discrete group $G$ with the property that the set $R_g$ of roots of $g$ in $G$ is finite for every $g\in G$, 
  $\Sub^c_G$ is closed in $\Sub_G$ (cf.\ \cite{PaS20}, Lemma 3.4). For such a group $G$, it is shown in Proposition 3.8 of \cite{PaS20} 
  that only finite order automorphisms act distally on $\Sub^c_G$, in case $G$ is finitely generated. Here, we study the expansivity 
  of actions of automorphisms of $G$ on $\Sub^c_G$ in the following. 
  
\begin{lemma}\label{nbd}
Let $G$ be a discrete group with the property that the set $R_g$ of roots of $g$ in $G$ is finite for all $g\in G$. Then the complement 
of any neighbourhood of $\{e\}$ in $\Sub^c_G$ is finite. Moreover, $G$ does not admit any
automorphism which acts expansively on $\Sub^c_G$ unless $G$ is finite. 
\end{lemma}

\begin{proof} If $G$ is finite, then $\Sub_G$ is finite, and the first assertion holds trivially and any automorphism of $G$ acts 
expansively on $G$. Now suppose $G$ is infinite. 
If possible, suppose there are infinitely many elements outside some open neighbourhood $U$ of $\{e\}$ in $\Sub^c_G$ namely, 
$G_{x_n}\notin U$, $x_n\neq x_m$ for all $m, n\in\N$. Note that $\Sub^c_G$, being closed in $\Sub_G$, is compact. 
Passing to a subsequence if necessary, we get that $G_{x_n}\to G_x$ for some $x\in G$, as $n\to\infty$. Then $G_x\notin U$, 
and hence $x\ne e$. As $G$ is discrete, there exists $n_0\in\N$ such that $x\in G_{x_n}$ for all $n\geq n_0$. 
It follows that $x={x_n}^{m_n}$ for some $m_n\in\Z$ and for all $n\geq n_0$. Replacing $x$ by $x^{-1}$ if necessary, we
may assume that $m_n\in\N$ for infinitely many $n$.  This leads to a contradiction as $x_n$'s are distinct and $R_x$ is finite.
Hence, given any neighbourhood of $U$ of $\{e\}$ in $\Sub^c_G$, $\Sub^c_G\setminus U$ is finite. 

Let $T\in\Aut(G)$. If possible, suppose $T$ acts expansively on $\Sub^c_G$ with an expansive constant 
$\epsilon>0$. Let $U=\{H\in\Sub^c_G\mid \db(H,\{e\})<\epsilon\}$,where $\db$ is the metric on $\Sub_G$. 
Since $G$ is infinite and $R_e$ is finite, there exists $x\in G$ which generates an infinite cyclic group. For every $k\in\N$, 
there exists $n_k\in\Z$ such that $\db(T^{n_k}(G_{x^k}),\{0\})>\epsilon$. Since $\Sub^c_G\setminus U$ is finite, we have that 
$T^{n_k}(G_{x^k})=T^{n_l}(G_{x^l})$ for infinitely many $k$ and $l$ with $k\ne l$. Here, $n_k\ne n_l$ if $l\ne k$, as $x$ has 
infinite order. Let $y=x^k$ for some fixed $k$. Then $y=T^{n_l-n_k}(x)^l$ for infinitely many $l$. 
Since $R_y$ is finite, we get that $\{T^{n_l-n_k}(x)\}_{l\in\N}$ is finite.  As each $n_l\ne n_j$, if $l,j\in\N$ and $l\ne j$, 
there exists $m\in\N$ such that $T^m(x)=x$. Now $T^m(G_{x^k})=G_{x^k}$ for infinitely many $k$. That is, $T^m$ has 
infinitely many fixed points in $\Sub^c_G$. This leads to a contradiction, due to Theorem 5.26 of \cite{Wa82}.
Therefore, $T$ is not expansive. 
\end{proof}

Note that Lemma~\ref{nbd}, in particular, implies that any discrete finitely generated infinite nilpotent group $G$ does not admit any 
automorphism that acts expansively on $\Sub^c_G$, as the set of roots of $g$ is finite for every $g\in G$. Such groups $G$ form a 
proper subclass of (discrete) polycyclic groups. If $G$ is any discrete polycyclic group, then every subgroup 
of it is finitely generated, and by Lemma 3.3 of \cite{PaS20}, $\Sub^c_G$ is closed in $\Sub_G$. 
The following theorem shows that such a $G$ does not admit any automorphism which acts expansively on $G$, 
unless $G$ is finite. The theorem will be useful in the proof of  Theorem~\ref{main-expa}. As noted before, the class of polycyclic 
groups is strictly larger than that of lattices in connected solvable Lie groups.

\begin{theorem}\label{poly-e}
Let $G$ be an infinite discrete polycyclic group and let $T\in\Aut(G)$. Then the $T$-action on $\Sub^c_G$ is not expansive. 
In particular, this holds when $G$ is any discrete solvable subgroup of a connected Lie group. 
\end{theorem}

\begin{proof}
Let $G'$ be a subgroup of finite index in $G$ which is strongly polycyclic. Passing to a subgroup of finite index if necessary, 
we may assume that $G'$ is $T$-invariant. Now we may replace $G$ by $G'$ and assume that $G$ is strongly polycyclic. 
Let $G(0)=G$ and for $n\in\N\cup\{0\}$, let $G(n+1)=[G(n),G(n)]$ be the commutator subgroup of $G(n)$. Each $G(n)$ is a 
characteristic subgroup of $G$. Since $G$ is solvable, there exists $k\in\N\cup\{0\}$ such that $G(k)\ne\{e\}$ and $G(k+1)=\{e\}$.
Here, $G(k)$ is an infinite strongly polycyclic abelian $T$-invariant group. Therefore, replacing $G$ by $G(k)$ if necessary, 
we may assume that $G$ is abelian. Now $G$ is isomorphic to $\Z^n$ for some $n\in\N$. It is easy to see that 
the set $R_g$ of roots of $g$ is finite in $G=\Z^n$ (see also Example 3.1.12 in \cite{He77}). Hence by Lemma~\ref{nbd}, 
the $T$-action on $\Sub^c_G$ is not expansive. 

Every discrete solvable subgroup of a connected Lie group is polycyclic by Lemma~\ref{solvable-Lie}\,(2), hence the second assertion 
follows from the first. \end{proof}

Recall that for a locally compact metrizable group $G$, $\Sub^a_G$ is the set of all closed abelian subgroups of $G$. It is closed in 
$\Sub_G$. The following result is already known for all connected Lie groups (cf.\  \cite{PrS20}, Theorem 3.1). Combining it with
Theorem~\ref{poly-e}, we get the following generalisation. 

\begin{corollary} \label{cor1}
Let $G$ be an infinite Lie group and let $T\in\Aut(G)$. If $G/G^0$ is polycyclic, then the $T$-action on $\Sub^a_G$ is not 
expansive. In particular, if $G$ is a closed subgroup of a connected Lie group $H$ such that $G$ is either solvable or normal in $H$, 
then the $T$-acton on $\Sub^a_G$ is not expansive.
\end{corollary}

\begin{proof} Let $G^0$ be the connected component of the identity in $G$. Then $G^0$ is a closed $T$-invariant subgroup of $G$, 
and hence it is a Lie group. If $G^0$ is nontrivial, by Theorem 3.1 of \cite{PrS20}, the $T$-action on $\Sub^a_{G^0}$, and hence on 
$\Sub^a_G$ is not expansive. If $G^0$ is trivial, then $G$ is discrete. If $G/G^0$ is polycyclic, i.e.\ $G$ is polycyclic, by Lemma~\ref{nbd}, 
the $T$-action on $\Sub^c_G$, and hence, on $\Sub^a_G$ is not expansive. 

Suppose $G$ is a closed subgroup of a connected Lie group $H$. If $G$ is solvable, then by Lemma~\ref{solvable-Lie}\,(2), 
$G/G^0$ is polycyclic. If $G$ is normal in $H$, then so is $G^0$, and $G/G^0$, being a discrete normal subgroup of the connected 
Lie group $H/G^0$, is central in $H/G^0$, and hence it is polycyclic by Lemma~\ref{solvable-Lie}\,(2). Therefore, in either
case, the second assertion follows from the first. 
\end{proof}

Now we are ready to prove one of the main results about expansivity related to automorphisms of lattices in connected 
Lie groups. Note that a lattice $\Ga$ in a compact Lie group is finite and hence all its automorphisms act expansively on $\Sub_\Ga$. 

\begin{theorem} \label{main-expa} A lattice $\Ga$ in a connected noncompact Lie group does not admit any automorphism which acts 
expansively on $\Sub^c_\Gamma$.
\end{theorem}

\begin{proof} Let $T\in\Aut(\Gamma)$. Let $G$ be a connected noncompact Lie group in which $\Gamma$ is a lattice. 
By Proposition~\ref{lattice-rad}\,$(a)$, $\Ga$ admits a largest solvable normal subgroup $\Ga_\rad$ which is polycyclic. Note that $\Ga_\rad$ 
is characteristic in $\Ga$, and hence it is $T$-invariant.  

If possible, suppose $T$ acts expansively on $\Sub^c_\Ga$. Then $T|_{\Ga_\rad}$ acts expansively on $\Sub^c_{\Ga_\rad}$. 
By Theorem~\ref{poly-e}, $\Ga_\rad$ is finite. By Proposition~\ref{quo-poly}\,(2), $\Ga$ has a subgroup of finite index $\Lam$ 
containing $\Ga_\rad$ such that the set $R_g$ of roots of 
$g$ in $\Lam/\Ga_\rad$ is finite, for every $g\in\Lam/\Ga_\rad$. This, together with the fact that $\Ga_\rad$ is finite, implies that the 
set $R_g$ of roots of $g$ in $\Lam$ is finite for every $g\in\Lam$. Passing to a subgroup of finite index if necessary, we may 
assume that $\Lam$ is $T$-invariant. Since $T$ acts expansively on $\Sub^c_\Ga$, and hence
on $\Sub^c_\Lam$, from Lemma~\ref{nbd} we get that $\Lam$ is finite, and hence $\Ga$ is also finite. 
This leads to a contradiction as $G$ is noncompact (cf.\ \cite{Ra72}, Remark 5.2\,(2) and Lemma 5.4).  
Hence, $T$ does not act expansively on $\Sub^c_\Ga$. 
\end{proof}

\noindent{\bf Acknowledgements:} R.\ Palit would like to acknowledge the CSIR-JRF research
fellowship from CSIR, Govt.\ of India. M.\ B.\ Prajapati would like to
acknowledge the UGC-JRF research fellowship from UGC, Govt.\ of India. R.\ Shah would like to acknowledge the 
MATRICS research grant from DST-SERB, Govt.\ of India which partially supported R.\ Palit while this work 
was carried out at JNU, New Delhi.

\bigskip
\begin{flushleft}
Rajdip Palit rajdip1729@gmail.com \\[1mm]
Manoj Prajapati manoj.prajapati.2519@gmail.com \\[1mm]
Riddhi Shah riddhi.kausti@gmail.com, rshah@jnu.ac.in \\[4mm]
School of Physical Sciences\\ 
Jawaharlal Nehru University\\
New Delhi 110067, India
\end{flushleft}

\end{document}